\documentclass[12pt]{article}  
\usepackage[english]{babel}
\usepackage[toc,page]{appendix}
\usepackage{fullpage}
\usepackage[top=1in, bottom=1.25in, left=0.9in, right=0.9in]{geometry}
\usepackage{etaremune}
\usepackage{enumitem}

\usepackage{amssymb,amsmath,amsthm}

\usepackage{hyperref}
\usepackage[T1]{fontenc}

\DeclareSymbolFont{rsfs}{U}{rsfs}{m}{n}
\DeclareSymbolFontAlphabet{\mathscrsfs}{rsfs}
\usepackage{mathrsfs}

\usepackage{graphicx}

\usepackage{url}

\hypersetup{
  colorlinks   = true, 
  urlcolor     = blue, 
  linkcolor    = blue, 
  citecolor   = red 
}

\usepackage[labelformat=empty]{caption}

\newtheorem{theorem}{Theorem}[section]
\newtheorem{lemma}[theorem]{Lemma}

\newtheorem{proposition}[theorem]{Proposition}
\newtheorem{corollary}[theorem]{Corollary}

\theoremstyle{definition}
\newtheorem{definition}{Definition}
\newtheorem{remark}[theorem]{Remark}

\numberwithin{equation}{section}

\usepackage{color}
\definecolor{bluegray}{rgb}{0.1, 0.1, 0.7}

\newcommand{\revv}[1]{{#1}}

\newcommand{\bea}{\begin{eqnarray}}
\newcommand{\eea}{\end{eqnarray}}
\newcommand{\<}{\langle}
\renewcommand{\>}{\rangle}

\newcommand{\wt}{\widetilde}
\newcommand{\op}{\text{op}}
\newcommand{\wh}{\widehat}

\def\bU{{\boldsymbol U}}

\def\Ann{{\rm Ann}}

\def\iid{\text{i.i.d.~}}

\def\sTV{\mbox{\tiny \rm TV}}

\def\Proj{{\sf P}}

\def\eps{{\varepsilon}}

\def\id{{\boldsymbol{I}}}

\def\cuP{\mathscrsfs{P}}

\def\ubeta{\underline{\beta}}
\def\bsigma{{\boldsymbol{\sigma}}}

\def\bs{{\boldsymbol{s}}}

\def\hS{\widehat{S}}
\def\hbs{{\boldsymbol{\widehat{s}}}}

\def\bW{{\boldsymbol{W}}}

\def\bx{{\boldsymbol{x}}}

\def\cF{{\mathcal F}}
\def\cG{{\mathcal G}}

\def\cC{{\mathcal C}}
\def\cE{{\mathcal E}}

\def\op{\mbox{\tiny\rm op}}

\newcommand{\ubq}{\bar{q}}
\newcommand{\lbq}{\underline{q}}

\def\wtH{\wt{H}}

\def\wtZ{\wt{Z}}

\def\bsig{{\boldsymbol {\sigma}}}

\def\brho{{\boldsymbol \rho}}

\def\by{{\boldsymbol y}}

\def\reals{{\mathbb R}}

\def\sT{{\sf T}}

\def\bv{{\boldsymbol{v}}}

\def\bx{{\boldsymbol{x}}}

\def\bB{\boldsymbol{B}}

\def\hbx{\boldsymbol{\hat{x}}}
\def\hby{\boldsymbol{\hat{y}}}

\def\Par{{\sf P}}
\def\de{{\rm d}}

\def\bX{\boldsymbol{X}}
\def\bY{\boldsymbol{Y}}
\def\bW{\boldsymbol{W}}
\def\prob{{\mathbb P}}

\def\<{\langle}
\def\>{\rangle}
\def\Tr{{\sf Tr}}

\def\Ball{{\sf B}}

\def\diam{{\rm diam}}

\def\cN{{\cal N}}

\def\cL{{\cal L}}

\def\by{{\boldsymbol{y}}}

\def\P{\mathbb{P}}
\def\cE{{\mathcal E}}

\def\cD{{\cal D}}

\def\bu{{\boldsymbol{u}}}
\def\b0{{\boldsymbol{0}}}

\def\sRS{\mbox{\tiny\rm RS}}
\def\sreg{\mbox{\tiny\rm reg}}

\def\obW{\overline{\boldsymbol W}}

\newcommand{\pl}{{\mbox{\rm\footnotesize pl}}}
\newcommand{\rd}{{\mbox{\rm\footnotesize rd}}}

\def\bfone{{\boldsymbol 1}}

\def\bG{{\boldsymbol G}}

\DeclareMathOperator*{\plim}{p-lim}

\def\cD{{{\mathcal D}}}

\def\cI{{\mathcal I}}
\def\cS{{\mathcal S}}

\def\sep{\mathrm{sep}}

\renewcommand{\b}{\mathbf{b}}

\def\lt{\left}
\def\rt{\right}

\def\la{\langle}
\def\ra{\rangle}

\def\eps{\varepsilon}

\def\bbE{{\mathbb{E}}}

\def\bbP{{\mathbb{P}}}
\def\bbR{{\mathbb{R}}}

\def\cF{{\mathcal{F}}}

\def\cN{{\mathcal{N}}}
\def\cP{{\mathcal{P}}}

\def\same{{\mathrm{same}}}
\def\Leb{{\mathrm{Leb}}}

\DeclareMathOperator*{\E}{\bbE}

\def\sph{\mathrm{sp}}

\def\bfone{{\boldsymbol 1}}

\author{Ahmed El Alaoui \and Andrea Montanari \and Mark Sellke}

\title{Shattering in Pure Spherical Spin Glasses}

\date{}

\begin{document}

\maketitle

\begin{abstract}
    We prove the existence of a shattered phase within the replica-symmetric phase of the pure spherical $p$-spin models for $p$ sufficiently large.
    In this phase, we construct a decomposition of the sphere into well-separated small clusters, each of which has exponentially small Gibbs mass, yet which together carry all but an exponentially small fraction of the Gibbs mass.
    We achieve this via quantitative estimates on the derivative of the Franz--Parisi potential, which measures the Gibbs mass profile around a typical sample. Corollaries on dynamics are derived, in particular we show the two-times correlation function of stationary Langevin dynamics must have an exponentially long plateau. We further show that shattering implies disorder chaos for the Gibbs measure in the optimal transport sense; this is known to imply failure of sampling algorithms which are stable under perturbation in the same metric.  
\end{abstract}



{\small
\tableofcontents
}

\section{Introduction}

Fix an integer $p\geq 2$ and let $\bG\in \lt(\bbR^N\rt)^{\otimes p}$ be an order $p$ tensor with \iid standard Gaussian entries $g_{i_1,\dots,i_p}\sim \cN(0,1)$. The pure $p$-spin Hamiltonian $H_N:\bbR^N\to\bbR$ is the random homogeneous polynomial
\begin{equation}
    \label{eq:def-hamiltonian}
    H_N(\bsig) = \frac{1}{N^{(p-1)/2}} \la \bG, \bsig^{\otimes p} \ra
    =
    \frac{1}{N^{(p-1)/2}}
    \sum_{i_1,\dots,i_p=1}^N g_{i_1,\dots,i_p}\sigma_{i_1}\dots\sigma_{i_p}.
\end{equation}

Consider the spherical Gibbs measure $\de \mu_{\beta}(\bx)\propto e^{\beta H_N(\bx)}~\,\de \mu_0(\bx)$, 
where $\de \mu_0(\bx)$ denotes the uniform measure on 
$\cS_N \equiv \{\bsig \in \bbR^N : \sum_{i=1}^N \bsig_i^2 = N\}$. It is now mathematically 
well established that $\mu_{\beta}$ undergoes a \emph{static} phase transition
 from replica symmetric to $1$-RSB at a value $\beta_c(p)$ given by
\begin{align}
 \beta_c^2(p)
    =
    \inf_{q\in [0,1]}
    \lt(\frac{1}{q^p}\log\lt(\frac{1}{1-q}\rt)-\frac{1}{q^{p-1}}\rt)\, .
    \label{eq:PspinExplicit}
\end{align}
(We refer to \cite{chen2013aizenman} which establishes 
the Parisi formula for a general mixed p-spin spherical model, and
to Appendix \ref{app:StaticTemperature} which derives the explicit formula \eqref{eq:PspinExplicit}.)

However another \emph{dynamical} (ergodicity breaking) transition has long been expected
at
\begin{equation}
\beta_d(p)=\sqrt{\frac{(p-1)^{p-1}}{p(p-2)^{p-2}}}<\beta_c(p)\, .\label{eq:FormulaBetaD}
\end{equation}
\revv{ Crisanti, Horner, Sommers  \cite{crisanti1993spherical} were the first to identify
this phase transition by studying the Langevin dynamics for this model (Langevin
dynamics is a diffusion on $\cS_N$ that is reversible with respect to $\mu_{\beta}$,
see definitions below). 
Define the correlation function $C(t) = \lim_{N \to \infty} \langle\bsig_0 , \bsig_t\rangle /N$ (assuming the limit exists) where $\bsig_0 \sim \mu_{\beta}$ and $\bsig_t$ follows Langevin dynamics initialized at $\bsig_0$ and let $q = \lim_{t \to \infty} C(t)$. 
These authors derived heuristically an integral-differential equation for the function
$C(t)$ (`dynamical mean-field theory' equation). By a heuristic
analysis of this equation, they showed that $q>0$ for $\beta>\beta_d(p)$
and  $q=0$ for $\beta<\beta_d(p)$. 
In other words, for $\beta>\beta_d(p)$, $\bsig_t$ retains long term memory of the initialization (meanwhile for $\beta \in (\beta_c(d),\beta_d(d))$,
$\<\bsig,\bsig'\>/N =o_N(1)$ for two independent configurations $(\bsig,\bsig')\sim 
\mu_{\beta}^{\otimes 2}$). 
We refer to \cite{cugliandolo1993analytical,cugliandolo1994out,cugliandolo2004course,ben2006cugliandolo}
for further physics background on this phase transition.
(Appendix~\ref{app:StaticTemperature} provides further details on 
formula \eqref{eq:FormulaBetaD}, and its generalization 
to mixed $p$-spin models.)}

For $\beta\in (\beta_d,\beta_c)$, $\mu_{\beta}$ is expected to exhibit \emph{shattering}: the measure  puts nearly all of its mass on well separated ``clusters'', each carrying an exponentially small fraction of the total mass. 
\revv{The connection between the dynamical phase transition discussed
above and the shattering phase transition was identified early on in the physics
literature, via the analysis of the solutions of TAP equations \cite{kurchan1993barriers,crisanti1995thouless},
and the introduction of the Franz-Parisi potential \cite{franz1995recipes}.
When shattering takes place, the heuristic argument goes, equilibrium dynamics is
initialized within one of the clusters, and remains confined to it for an exponentially 
large time. The plateau in the correlation function at $q$ corresponds to the typical 
overlap between two configurations in the same cluster being $\<\bsig,\bsig'\>/N=q+o_N(1)$.}

Important progress towards making rigorous this scenario was achieved in \cite{arous2021shattering}. For $p\geq 4$
 and $\beta\in (\beta_d,\beta_c)$ close enough to $\beta_c$, \cite[Theorem 2.4]{arous2021shattering} 
 gives a decomposition where the clusters are centered around the TAP-states
  (critical points of the TAP free energy), and  which attains the full free energy of the model, 
  i.e., the decomposition covers a $e^{-o(N)}$ fraction of the Gibbs mass.
  \revv{While such a shattering result is used to prove exponentially large relaxation
  time, it is not strong enough to imply a plateau in the equilibrium correlation function
  (because the initial configuration $\bsig_0$ is not guaranteed to be in one of the clusters.)}

Note that for all $p\geq 3$ one easily verifies 
\begin{equation}
\label{eq:beta-d-bounded}
    1 < \beta_d(p) < 2\, .
\end{equation}
Further for $p$ large we have
\begin{equation}
\beta_d(p) =\sqrt{e}+O(1/p)\, ,\;\;\;\;\; \beta_c(p) = \sqrt{\log p}\cdot
(1+o_p(1))\, ,
\end{equation}
so that the shattered phase is conjectured to occupy a diverging
interval of temperatures for $p$ large. 

In this paper we prove that 
shattering indeed occurs for all $p$ large enough in a interval $\beta\in (\ubeta(p),\beta_c(p))$,
where $\ubeta(p)=\Theta(1)$ for large $p$. Further, we establish that shattering holds in a stronger sense
than \cite{arous2021shattering}: our decomposition covers at least $1-\exp(-\Theta(N))$
of the Gibbs measure. This has immediate `observable' consequences on the behavior
of Langevin dynamics at equilibrium. \revv{Crucially, it implies the plateau in the correlation function $C(t)$ conjectured in the physics literature.}

Our approach to shattering is direct and inspired by
a technique developed to prove  shattering in
random constraint satisfaction problems
in theoretical computer science 
\cite{achlioptas2008algorithmic}. 
We show that for a typical Gibbs sample $\bsig\sim\mu_{\beta}$, there exist radii 
$\Delta_1\ll \Delta_2\ll 1$ such that nearly all of the Gibbs mass in
 $B_{\Delta_2}(\bsig)$ lies in $B_{\Delta_1}(\bsig)$ (the balls are with respect 
 to the normalized $\ell_2$ distance). 
A technical challenge lies in the fact 
that the centers of these balls have a distribution 
that depends on $\bG$ in a complex manner.
We circumvent this difficulty via a contiguity argument which allows us to switch to a
 simpler \emph{planted distribution} in which $\bsig$ is drawn from the uniform measure on the
  sphere. This technique is relatively standard and has been used in~\cite{achlioptas2008algorithmic} 
  in a zero temperature setting for random constraint satisfaction problems. 
Then, precise control of certain free energy derivatives is obtained via the Parisi formula.

\section{Main results}

\subsection{Shattering theorems}

We will present two theorems establishing shattering in the spherical 
$p$-spin model. The first holds under an explicit condition of the form
$\beta\in (\ubeta(p),\beta_c(p))$ and implies a decomposition of the Gibbs measure into states that take
the form of spherical caps. The second theorem holds under a certain condition 
on the so-called Franz--Parisi potential \cite{franz1995recipes}, which we expect to be sharp although not always easy to verify. 
The latter condition is actually more general (i.e., the second theorem encompasses the first), except that the states may become somewhat more complicated than spherical caps.

We begin by introducing some useful terminology. We endow $\cS_{N}$ with the Euclidean distance
$d(\bsig_1,\bsig_2):=\|\bsig_1-\bsig_2\|_2$, and will occasionally use
$\|\bx\|_N:=\|\bx\|_2/\sqrt{N}$ for the normalized $\ell_2$ norm of $\bx\in\reals^N$. 
As usual, for sets $S_1,S_2\subset \cS_N$, we write
$d(S_1,S_2)=\inf\limits_{\bsig_1\in S_1,\bsig_2\in S_2}\|\bsig_1-\bsig_2\|_2$,
as well as $d(\bsig_1,S_2):= d(\{\bsig_1\},S_2)$ for $\bsig_1\in\cS_N$.
Finally, we define the $r$-neighborhood of $S\subseteq \cS_{N} $ by
$\Ball_{r}(S) = \{\bsig \in \cS_N \,:\, d(\bsig,S) \le r\sqrt{N}\}$ and write 
$\Ball_{r}(\bsig)=\Ball_r(\{\bsig\})$.
%
\begin{definition}\label{def:Shattering}
We say the sets (clusters) $\{\cC_1,\dots,\cC_M\}$ with $\cC_m\subseteq\cS_N$ are a \textbf{shattering decomposition} 
of the probability measure $\mu_{\beta}$, with parameters $(c,r,r_b,s)$, if the following hold:
 \begin{enumerate}[label={\sf S\arabic*},ref={\sf S\arabic*}]
        \item
        \label{it:small-prob}
        Each cluster has exponentially small probability:
        \[
        \max_{1\leq m\leq M}\,\, \mu_{\beta}(\cC_m)\leq e^{-cN}\, .
        \]
        \item 
        \label{it:small-diam}
        Each cluster is geometrically small:
        \[
        \diam(\cC_m)\leq r \sqrt{N} \quad\forall~m\in [M]\, .
        \]
         \item 
        \label{it:clusters-isolated}
        Each cluster is surrounded by a bottleneck:
        \[
        \mu_{\beta}(\cC_m)\geq (1-e^{-cN})\mu_{\beta}\big(\Ball_{r_b}(\cC_m)\big) \, .
        \]
        \item 
        \label{it:clusters-separate}
        Distinct clusters are well-separated: 
        \[
        \min_{1\leq m_1<m_2\leq M} d\big(\cC_{m_1},\cC_{m_2}\big)\geq s \sqrt{N} \, .
        \]
        \item 
        \label{it:clusters-cover}
        Collectively, the clusters carry nearly all of the Gibbs mass:
        \[
        \mu_{\beta}\lt(
        \bigcup\limits_{m=1}^M \cC_m
        \rt)
        \geq 1-e^{-cN}\, .
        \] 
        \end{enumerate} 
\end{definition}
We are now in position to state the first theorem.
\begin{theorem}
\label{thm:main}
    For $p$ large enough, there exists a non-empty interval 
    $I_p=(\ubeta(p),\beta_c(p))\subseteq (\beta_d(p),\beta_c(p))$, with $\ubeta(p) = C$
 independent of $p$, such that for any $\beta\in I_p=(\ubeta(p),\beta_c(p))$, there exist $\Delta_1,\Delta_2,c>0$ with 
    $\Delta_2\geq \Omega(\sqrt{\beta}\Delta_1)$ such that the following holds.
    
    With probability $1-e^{-cN}$, there exists a finite collection
     $\lt\{\cC_1,\dots,\cC_M\rt\}$ of spherical caps  $\cC_m\subseteq\cS_N$ depending on $\bG$ 
      that form a shattering decomposition 
of the probability  measure $\mu_{\beta}$, with parameters $(c,\Delta_1,\Delta_2,\Delta_2)$.

        Furthermore, if $p$ is even, we may assume the clustering is origin-symmetric, i.e., the
         antipodal set $-\cC_i$ of any cluster is also a cluster.
\end{theorem}

\begin{remark}
\label{rem:Delta-shrinks-with-p}
    The distances in the above construction scale as $\Delta_1=\Theta\lt(1/\sqrt{\beta p}\rt)$ and $\Delta_2=\Theta\lt(1/\sqrt{p}\rt)$. The fact that $\Delta_1\ll \Delta_2$ for large $\beta$ strengthens the notion of shattering, as it means that different clusters are at a much larger distance apart than their diameters.
\end{remark}

\begin{remark}
    It is useful to compare our result to previously used notions of shattering.
    As previously mentioned, the influential work \cite{achlioptas2008algorithmic} proved a similar statement for constraint satisfaction problems also via contiguity with a planted model
    (see also \cite{biroli2000variational,mezard2002analytic,mezard2005clustering,
    achlioptas2006solution} for earlier results).
    As their work effectively takes place at zero temperature, instead of Assertion~\ref{it:clusters-isolated} they  require that all paths connecting distinct clusters must incur a energy cost linear in $N$  along the way. The definition of shattering we use here is adapted to the positive temperature setting and is similar to a definition given in~\cite{krzakala2007gibbs}
    (which in turns formalizes ideas from a long line of work, e.g.,
    \cite{franz1995recipes,monasson1995structural}). In the latter paper, the term \emph{clustering} is used to refer to the same phenomenon.
    From a technical perspective, \cite{achlioptas2008algorithmic} essentially proceeds via first and second moment computations, while our proof requires quantitative control of a certain Parisi measure in the planted model (see Proposition~\ref{prop:Franz-Parisi-deriv-positive}).

    The authors of \cite{arous2021shattering} consider the same model and show for $\beta$ sufficiently close to $\beta_c(p)$, the onset of a weaker version 
    of shattering where the decomposition carries a part of the Gibbs measure which is not exponentially small.
    They do this by computing the free energy of
    Thouless-Anderson-Palmer (TAP) states, and count them via the Kac-Rice formula.
    On the other hand, their work has the advantage of applying for all $p\geq 4$.

    Shortly after the initial posting of this paper, \cite{gamarnik2023shattering} independently showed the Ising pure $p$-spin model exhibits similar shattering behavior when $\beta\in (\sqrt{\log 2},\sqrt{2\log 2})$ and $p\geq P(\beta)$, by estimating certain first and second moments.
\end{remark}

As anticipated, our second theorem holds under certain assumptions on the behavior
of the so called Franz--Parisi potential.
With $\bsig\sim \mu_{\beta}$ a configuration distributed 
according to the Gibbs measure, the Franz--Parisi potential is defined by the following restricted free energy, for $q\in (-1,1)$:
\begin{equation}
\label{eq:FP-def}
    \cF_{\beta}(q) :=
    \lim_{\eps\downarrow 0}
    \plim_{N\to\infty} 
    \frac{1}{N}
    \log 
    \int e^{\beta H_N(\bsig')}
    1_{(\bsig,\bsig')\in \cS_N(q,\eps)}
    \de\mu_0(\bsig')\, ,
\end{equation}
where $\cS_N(q,\eps)$ is the set of pairs $(\bsig,\bsig')$ with overlap approximately $q$: $|\langle \bsig,\bsig'\rangle/N-q|\le \eps$.
For $\beta<\beta_c(p)$, existence of the limits and a Parisi-type variational 
formula for $\cF_{\beta}(q)$ are stated in Proposition \ref{prop:Franz-Parisi-formula} below.
\begin{theorem}
\label{thm:improved-clustering}
Assume $p\ge 3$ and $\beta<\beta_c(p)$ satisfy the non-monotonicity condition:
\begin{enumerate}[label={\sf FP},ref={\sf FP}]
  \item 
  \label{it:nonmonotone} 
  $\cF_{\beta}(q)$ is strictly increasing on some non-empty interval
   $(\lbq,\ubq)\subseteq [0,1]$, with $\lbq\geq 0.9999$.
\end{enumerate}

 Then for sufficiently small $\delta,c>0$, with probability $1-e^{-cN}$, there exist path-connected 
clusters $\cC_1,\dots,\cC_M\subseteq\cS_N$   that form a shattering decomposition 
of the probability  measure $\mu_{\beta}$, with parameters $(c,7\sqrt{1-\ubq},\delta,\delta)$.

   Moreover if $p$ is even then the clustering may be taken symmetric, so that the antipodal set $-\cC_i$ of any cluster is also a cluster.
\end{theorem}

\revv{
\begin{remark}
Note that the radius parameter $r$ in the shattering decomposition 
established on the last theorem is $r=7\sqrt{1-\ubq}$.
We require the lower bound condition $\lbq\ge 0.9999$ 
to be very close to $1$,  so that  $\sqrt{1-\lbq}$ is also sufficiently small.
This bound will also be used in proving Theorem~\ref{thm:disorder-chaos}.
We note that some lower bound on $\lbq$ is needed to ensure disjointness of clusters (e.g., clusters with orthogonal centers).
\end{remark}
}

\revv{
\begin{remark}
While our results concern the pure $p$-spin spherical model of 
Eq.~\eqref{eq:def-hamiltonian}, we believe our techniques extend
to a suitable subset of mixed models as well.
In these models, the Hamiltonian has covariance
$\E[H_N(\bsig_1)H_N(\bsig_2)] = N\xi(\<\bsig_1,\bsig_2\>/N)$ for a 
polynomial $\xi$ (pure $p$-spin models correspond to 
$\xi(q) = q^p$). The main property of $\xi(q)$ that we use in our proofs 
is the fact that it has a large derivative in a neighborhood of $q=1$.
\end{remark}
}

\subsection{Consequences of shattering: I. Langevin dynamics}

As already discussed in \cite{ben2018spectral}, shattering has immediate
consequences for the Langevin dynamics for $\mu_{\beta}$. Langevin dynamics
is defined by the following diffusion on $\cS_N$:
\begin{align}
&\bsig_0\sim \mu_{\beta}(\,\cdot\,)\, ,\\
&\de\bsig_t = \left(\beta \nabla_{\sph} H_N(\bsig_t)-\frac{N-1}{N}\bsigma_t\right)\de t
+\Proj^{\perp}_{\bsig_t}\sqrt{2}\de\bB_t\, ,
\end{align}
where 
\[
\nabla_{\sph} \phi(\bsig) := \Proj^{\perp}_{\bsig} \nabla \phi(\bsig)  \, 
\]
denotes the spherical gradient of the function $\phi$, i.e. the usual gradient projected onto the space orthogonal to $\bsig$: $\Proj^{\perp}_{\bsig}=\id-\bsig\bsig^{\sT}/N$. It is well known, and easy to verify, that this
 is a Markov process, ergodic and reversible with respect to $\mu_{\beta}$.
 \revv{
 Its infinitesimal generator $\cL_{\beta}$ is given by 
 \begin{equation}
 \label{eq:L-beta}
\cL_{\beta}
=
\Delta_{\sph} +\beta\la \nabla_{\sph}\,\cdot\,,\nabla_{\sph}H_N\ra \, ,
 \end{equation}
 where $\Delta_{\sph}$ is the spherical Laplacian.
 }

The first consequence we state was already established in  \cite{ben2018spectral}: Langevin
dynamics fails to be rapidly mixing, and has exponentially many eigenvalues which are exponentially small.
Here we extend this corollary to the domain $(\beta,p$) covered by 
Theorems~\ref{thm:main} and \ref{thm:improved-clustering}.
The following standard estimate will be useful; we note that the scaling in $p$ appears in Proposition~\ref{prop:Franz-Parisi-deriv-positive}.

\revv{
\begin{proposition}
\label{prop:norm-bound}
Let $H_N$ be the spherical $p$-spin Hamiltonian.
Then there exists a universal constant $C$ independent of $p$
such that with probability $1-e^{-cN}$, 
    \begin{equation}
    \label{eq:C-bounded}
    \sup_{\bsig\in\cS_N}\|\nabla^j H_N(\bsig)\|\leq \frac{C N^{1-\frac{j}{2}}}{10} p^j\sqrt{\log p}\, .
    \end{equation}
Here we use the injective tensor norm given by (c.f. \cite{banach1938homogene})
\[
    \|\nabla^j H_N(\bsig)\|
    :=
    \sup_{\|\bx^1\|=\dots=\|\bx^j\|=1}
    \la \nabla^j H_N(\bsig),\bx^1\otimes\dots\otimes \bx^j\ra
    =
    \sup_{\|\bx\|=1}
    \la \nabla^j H_N(\bsig),\bx^{\otimes j}\ra \, .
\]
\end{proposition}
}

\begin{proof}
This follows by \cite[Lemma 2.1]{richard2014statistical} and Borell-TIS when $j=0$.
The extension to larger $j$ is immediate by homogeneity.
\end{proof}

\begin{corollary}
\label{cor:spectral-gap}
    For $(\beta,p)$ either as in Theorem~\ref{thm:main} 
    or in Theorem \ref{thm:improved-clustering}, with probability $1-e^{-cN}$ 
    the first $e^{cN}$ eigenvalues of $\cL_{\beta}$ are at most $e^{-cN}$.
\end{corollary}
\begin{proof}
Let $\Delta=\Delta_2$ under the conditions of 
Theorem \ref{thm:main} or $\Delta=\delta$ under
Theorem \ref{thm:improved-clustering}.
Let $\gamma:\bbR\to [0,1]$ be an increasing smooth function with $\gamma(0)=0$ and $\gamma(1)=1$.
For each cluster $\cC_{m}$ in the decomposition of these theorems, define the test function
\begin{align}
\phi_{m}(\bsig) :=\begin{cases}
1 & \mbox{ if $\bsig\in\cC_m$,}\\
\gamma\left(1-\frac{3d(\bsig,\cC_m)}{\Delta\sqrt{N}}\right)& \mbox{ otherwise.}
\end{cases}
\end{align}
Now using Theorem \ref{thm:main} \revv{or \ref{thm:improved-clustering}, and Proposition~\ref{prop:norm-bound},} it is easy to see that for each $m$, and some constant $c>0$,
\begin{align}
\frac{\<\phi_m,(-\cL_{\beta})\phi_m\>_{\mu_{\beta}}}
{\|\phi_m\|_{\mu_\beta}^2} \le e^{-cN}\, ,\label{eq:Dir}
\end{align}
where the norm and scalar product are in $L^2(\cS_N,\mu_{\beta})$.
\revv{
Indeed we have $\|\phi_m\|_{\mu_\beta}^2\geq \Omega(\mu_{\beta}(\cC_m))$ for the denominator. 
Meanwhile using \eqref{eq:L-beta} and Proposition~\ref{prop:norm-bound} shows that with probability $1-e^{-cN}$, the numerator is 
\[
\<\phi_m,(-\cL_{\beta})\phi_m\>_{\mu_{\beta}}
\leq 
O\big(\mu_{\beta}(\Ball_{\Delta/3}(\cC_m))\big).
\]
Namely the integral defining the numerator is nonzero only on $\mu_{\beta}(\Ball_{\Delta/3}(\cC_m))$, and is uniformly bounded on this set under the event of Proposition~\ref{prop:norm-bound}.
Assertion~\ref{it:clusters-isolated} of Definition~\ref{def:Shattering} thus implies \eqref{eq:Dir} as claimed.
}

Finally $\phi_m,\phi_{m'}$ have disjoint support for $m\neq m'$, so
Eq.~\eqref{eq:Dir} holds for all $f\in {\rm span}(\phi_m: 1\le m\le e^{cN})$. The claim follows by the variational characterization of 
eigenvalues.
\end{proof}

In spin glass physics, the two-time correlation function is more often considered a signature of slow dynamics (for the simple reason that eigenvalues
of the generator cannot be observed in a physical system). The next corollary shows that,
in the shattered phase, this correlation indeed exhibits an exponentially 
long plateau. We note that, for proving this fact, it is crucial
that we establish that the shattering decomposition covers all but a vanishing fraction of the Gibbs measure, i.e. Assertion \ref{it:clusters-cover} of Definition \ref{def:Shattering}. The
shattering result of \cite{arous2021shattering} is not sufficient for this purpose.
(On the other hand, a similar result was proven in \cite{montanari2006rigorous} for
a spin glass model defined on a sparse hypergraph.)
\begin{corollary}
\label{cor:CorrelationFunction}
Let $(\bsigma_t)_{t\ge 0}$ be a stationary trajectory 
for the Langevin dynamics.
Let $\Delta = \Delta_1+\Delta_2$ for $(\beta,p)$  under the assumptions of  Theorem~\ref{thm:main},
or $\Delta = 7\sqrt{1-\ubq}+\delta$ under the assumptions of Theorem \ref{thm:improved-clustering}.
Then there exists $c>0$ such that
%
\begin{align}
  \prob\left(
    \inf_{0\leq t\leq e^{cN}} 
    \frac{1}{N}\<\bsigma_0,\bsigma_t\>\ge 1-\Delta^2\right)
    \ge 1-e^{-cN}\, .
\end{align}
In particular, letting $C_N(t):=\E\<\bsigma_0,\bsigma_t\>/N$
denote the correlation function, for all $N\ge N_0(\beta,p)$
we have $C_N(t)\ge 1-2\Delta^2$
for all $t\leq e^{cN}$.
\end{corollary}
\begin{proof}
Under the conditions of Theorem~\ref{thm:main}, we let $r=\Delta_1$, $s=\Delta_2$,
and under the conditions of Theorem \ref{thm:improved-clustering}, let 
$r=7\sqrt{1-\ubq}$, $s=\delta$,
    
    Let $\cG_t:= \sigma(\bB_s:s\le t)$ be the filtration of Brownian
    motion.
    By e.g., \cite[Lemma 2.1]{sellke2023threshold}, for any $\eps>0$ there exists 
    $\rho=\rho(\eps,p,C)$ (playing the role of $s$ therein) such that the following 
    continuity property holds whenever \eqref{eq:C-bounded} holds.
    If $\tau$ is any stopping time for the Langevin dynamics then
    \begin{equation}
    \label{eq:langevin-continuous}
    \bbP\lt(
    \sup_{t\in [\tau,\tau+\rho]}
    \|\bsig_t-\bsig_{\tau}\|\leq \eps\sqrt{N}
    ~\Big|~\big(\bG,\cG_{\tau}\big)
    \rt)
    \geq 1-e^{-cN}.
    \end{equation}

    We may restrict our attention to the exponentially likely $\bG$-measurable event $\cE$
    that the shattering decomposition of Definition \ref{def:Shattering} exists and \eqref{eq:C-bounded} holds.
    Let $\overline{\cC}=\cup_{i=1}^M \cC_i$, $T=e^{cN/3}$ and denote by $\mu_{\Leb}$
     standard Lebesgue measure on $\bbR$.
    Since $\bsig_t\sim \mu_{\beta}$ for each $t\in [0,2T]$ by stationarity, we obtain,
    on $\cE$,
    \begin{equation}
    \label{eq:few-exceptional-times}
        \bbE \lt[\mu_{\Leb}\lt(
        \lt\{
        t\in [0,2T]~:~\bsig_t\notin \overline{\cC}
        \rt\}
        \rt)\big|\bG
        \rt]
        \leq 
        2Te^{-cN}\leq e^{-cN/2}\, .
    \end{equation}
        
    Let $\tau\geq 0$ be the first time at which $d(\bsig_{\tau},\overline{\cC})\geq s\sqrt{N}/3$. 
    Using \eqref{eq:langevin-continuous}, on $\cE$ we have, for a positive constant $\rho_0$,
    \begin{equation}
    \label{eq:many-exceptional-times}
        \bbE \lt[\mu_{\Leb}\lt(
        \lt\{
        t\in [0,2T]~:~\bsig_t\notin \overline{\cC}
        \rt\}
        \rt)
        ~\big|~
        \big(\bG,\cG_{\tau}\big)
        \rt]
        \geq 
       \rho_0\bfone_{\tau\leq T} \, .
    \end{equation}
    Taking expectations and combining with \eqref{eq:few-exceptional-times}, it follows that,
    for a constant $c'>0$,
    \begin{equation}
    \label{eq:tau-lage}
    \bbP(\tau\leq T)\leq e^{-2c'N}/\rho_0\leq e^{-c'N}.
    \end{equation}
    
    Finally on the event $\tau>T$, continuity of $\bsig_t$ implies that $(\bsig_t)_{t\in [0,T]}$ 
    stays within a single connected component of the $ s\sqrt{N}/3$-neighborhood of
     $\overline{\cC}$.
     Applying Theorem \ref{thm:main} or Theorem \ref{thm:improved-clustering}
    Part~\ref{it:clusters-separate} of Definition \ref{def:Shattering} 
    then implies 
    that 
    $\bsig_t$ stays inside the $s\sqrt{N}/3$-neighborhood of a single cluster $\cC_i$. 
    This neighborhood has diameter at most $(r+s)\sqrt{N}$ by Part~\ref{it:small-diam}.
    Therefore if $\tau>T$ and the decomposition Definition \ref{def:Shattering}  both hold, then
    \[
   \frac{1}{N} \la \bsig_0,\bsig_t\ra
    =
    1-\frac{\|\bsig_0-\bsig_t\|^2}{2N}
    \geq 
    1-(r+s)^2
    ,\quad\forall~t\in [0,T]\, .
    \]
    This completes the proof (modulo adjusting the value of $c$).
\end{proof}

\subsection{Consequences of shattering: II. Disorder chaos}

We also prove that a transport notion of disorder chaos follows from either Theorem~\ref{thm:main} or Theorem~\ref{thm:improved-clustering}.
Define the normalized Wasserstein distance, given for probability measures $\mu,\nu$ on $\cS_N$ by
\[
    W_{2,N}(\mu,\nu)^2 =  \inf_{\pi \in \Pi(\mu,\nu)} \frac{1}{N} 
    \E_{\pi} \Big[\big\|\bX - \bY\big\|_2^2\Big],
\]
where the infimum is over all couplings $(\bX,\bY) \sim \pi$ with marginals $\bX \sim \mu$ and 
$\bY \sim \nu$.  
Moreover, define the $(1-\eps_N)$-correlated disorder
\begin{equation}
\label{eq:correlated-disorder}
\bG^{\eps_N}
=
(1-\eps_N)\bG+\sqrt{2\eps_N-\eps_N^2}\bW
\end{equation}
where $\bW$ is a i.i.d. copy of $\bG$, and let $\mu_{\beta}^{\eps_N}$ be the corresponding Gibbs measure.  
We say that \emph{transport disorder chaos} holds for the random pair $(\mu_{\beta},\mu_{\beta}^{\eps_N})$ of Gibbs measures if
\begin{equation}
\label{eq:def-disorder-chaos}
\liminf_{N\to\infty}\,\,
    \bbE
    \lt[
    W_{2,N}\lt(\mu_{\beta},\mu_{\beta}^{\eps_N}\rt)
    \rt]
    >0\, .
\end{equation}
As discussed below, this is a somewhat stronger notion of chaos than the overlap-based one used in  \cite{chatterjee2009disorder}.
We say that \emph{transport disorder stability} holds if 
$\lim_{N\to\infty}
    \bbE
    \lt[
    W_{2,N}\lt(\mu_{\beta},\mu_{\beta}^{\eps_N}\rt)
    \rt]
    =0$.

\begin{theorem}
\label{thm:disorder-chaos}
    Let $(\beta,p)$ be as in Theorem~\ref{thm:main} or Theorem~\ref{thm:improved-clustering}. Let $\mu_{\beta}$ and $\mu_{\beta}^{\eps_N}$ be two Gibbs measures for correlated Gaussian disorders $\bG$, $\bG^{\eps_N}$
    as in Eq.~\eqref{eq:correlated-disorder}. 
    If $\eps_N = \Omega(1/N)$, then transport disorder chaos holds for $(\mu_{\beta},\mu_{\beta}^{\eps_N})$.
\end{theorem}

 As explained in \cite{alaoui2022sampling,alaoui2023sampling}, transport disorder chaos poses a barrier to 
natural sampling algorithms which are stable in an appropriate Wasserstein sense.
These include the Lipschitz algorithms studied in \cite{gamarnik2019overlap,huang2021tight} as well as low-degree polynomials (via e.g. \cite[Lemma 3.4]{gamarnik2020optimization}).
More generally, it seems to be a natural candidate property to characterize 
hardness of sampling from random high-dimensional probability measures.
In the case of the spherical $p$-spin model, we expect a (possibly weaker) version of Theorem~\ref{thm:disorder-chaos} 
to hold for all $\beta\in (\beta_d,\beta_c)$. 
Analogous properties which are tailored to optimization instead of sampling can be found in
\cite{gamarnik2019overlap,gamarnik2020optimization,huang2021tight,huang2023algorithmic}. 

We also observe that the scaling of $\eps_N$ in Theorem~\ref{thm:disorder-chaos} is best possible in some generality. Namely for $\eps_N= o(1/N)$, the perturbed Gibbs measure is close to the original in total variation, which is a very strong form of disorder stability.
\begin{proposition}
\label{prop:chaos-converse}
    For any spherical or Ising mixed $p$-spin model, if \revv{$\beta_N\leq O(1)$ and} $\eps_N = o(1/N)$ then 
    \[
    \lim_{N\to\infty}
    \bbE
    \lt[
    \|\mu_{\beta_N}-\mu_{\beta_N}^{\eps_N}\|_{\sTV}
    \rt]
    =0\, .
    \]
    As a consequence, transport disorder stability holds for $(\mu_{\beta_N},\mu_{\beta_N}^{\eps_N})$.
\end{proposition}

The onset of disorder chaos at $\eps_N\asymp 1/N$ was previously obtained in \cite{subag2017geometry} for pure spherical spin glasses at \emph{sufficiently low} temperature, which admit a different $1$-RSB decomposition into clusters of macroscopic Gibbs weight.
Proposition~\ref{prop:chaos-converse} simplifies and strengthens the analogous result \cite[Proposition 43]{subag2017geometry}.

As pointed out above, our notion of chaos is stronger than the one of \cite{subag2017geometry}, which we will refer to as ``overlap disorder chaos''. According to the latter,
the correlated measures $\mu_{\beta}$, $\mu_{\beta}^{\eps_N}$ 
exhibit disorder chaos if the overlap $\<\bsig,\bsig'\>/N$ concentrates
around a non random value $q_0$ 
when $(\bsig,\bsig')\sim \mu_{\beta}\otimes\mu_{\beta}^{\eps_N}$.
\revv{(Here and below we write $\mu\otimes\mu'$ for the product of probability measures $\mu$ and $\mu'$.)}
This overlap-based definition of chaos has been standard in the probability and statistical physics literatures, e.g., see \cite{bray1987chaotic,krzkakala2005disorder,chatterjee2009disorder} and is technically
convenient. However, it can be misleading: for instance
$\mu_{\beta}=\mu_{\beta}^{\eps_N}=\mu_0$ (the uniform measure)
exhibits overlap disorder chaos.

An important advantage of our transport definition is that it behaves nontrivially even when overlaps converge in probability to $0$.
In \cite{alaoui2022sampling}, we gave an explicit sampling algorithm proving that the Sherrington--Kirkpatrick model is transport disorder stable in the entire replica-symmetric phase $\beta<1$ (the sharp threshold following being proven in subsequent work by Celentano \cite{celentano2022sudakov}).
For completeness, Appendix~\ref{sec:trivial-overlaps} shows that
overlap disorder chaos holds in the entire replica-symmetric phase, for all mixed $p$-spin models without external field. This further justifies our transport-based definition of chaos.

\section{Partition function estimates}

We define the limiting annealed and quenched free energies,
 which agree for all $\beta\leq \beta_c$: 
%
\begin{equation}
\label{eq:free-energy-def}
\begin{aligned}
    F_{\beta} :=  \lim_{N\to\infty}
    \frac{1}{N}
    \bbE
    \log  
    \int e^{\beta H_N(\bsig)}
    &\de \mu_0(\bsig)  \, , ~~~~ F^{\Ann}_{\beta}
    :=
    \lim_{N\to\infty}
    \frac{1}{N}
    \log \bbE 
    \int e^{\beta H_N(\bsig)}
    \de \mu_0(\bsig) \, ,
    \\
    F_{\beta} &= F^{\Ann}_{\beta} = \frac{\beta^2}{2} \,, ~~~ \forall \beta \le\beta_c(p)\, .
\end{aligned}
\end{equation}
The identity of annealed
and quenched free energy for $\beta \le\beta_c(p)$ follows from Parisi's formula
\cite{talagrand2006free,chen2013aizenman} which we recall below in the form given by Crisanti-Sommers 
\cite{crisanti1992spherical} and recently studied in
\cite{jagannath2018bounds}. 

This formula is a strictly convex lower semicontinuous (in the weak$^*$ topology) functional 
\revv{$\Par:\cuP([0,1))\to \reals$
on the space of probability measures with support contained inside $[0,1)$}:
\begin{align}
\Par_{\beta}(\zeta;\xi) 
&=
\frac{1}{2}\left(\int_{0}^1 \beta^2\xi'(t)
\zeta([0,t])
\de t 
+
\int_0^{\hat{q}} \frac{\de t}{\phi_{\zeta}(t)}  +\log(1-\hat{q})\right)\, ,
\label{eq:CrisantiSommers}
\\
\phi_{\zeta}(t) & = \int_{t}^1\zeta([0,s])\de s\, ,~~~ \hat{q} = \inf\{s\in [0,1) \, :\, \zeta([0,s])=1\}\revv{\,<1}\, .
\nonumber
\end{align}
Here $\xi:\bbR\to\bbR$ is an analytic function with non-negative Maclaurin coefficients and radius of convergence strictly larger than $1$, and $\xi(0)=0$.
For the case of interest (pure $p$-spin) $\xi(t)= t^p$.

The relation between $\Par_{\beta}(\zeta;\xi)$
and the free energy is given by the following variational principle \cite{talagrand2006free,chen2013aizenman}.
Let $F_{\beta}(\xi)$ be the asymptotic quenched free energy 
defined as in Eq.~\eqref{eq:free-energy-def}, except that now $H_N$ is replaced by 
$H^{(\xi)}_N$, that is
a general centered Gaussian process  with covariance 
\begin{align}
\E\big[H^{(\xi)}_N(\bsig_1) H^{(\xi)}_N(\bsig_2)\big]
=N\xi(\<\bsig_1,\bsig_2\>/N)\, .\label{eq:Hxi}
\end{align}
Explicitly
\begin{equation}\label{eq:_mixed}
F_{\beta}(\xi) := \lim_{N\to \infty}\frac{1}{N}
    \E\log 
    \int e^{\beta H_N^{(\xi)}(\bsig)}
    \de\mu_0(\bsig)\, .
    \end{equation}
Then we have
\begin{align}
F_{\beta}(\xi) = \min_{\zeta\in\cuP([0,1])}\Par(\zeta;\xi)\, .\label{eq:GeneralVariational}
\end{align}
Below we will make use of this general variational principle.
\revv{Recall from just below \eqref{eq:FP-def} that $\cS_N(q,\eps)$ is} the set of pairs of point with overlap $q$: 
\[
    \cS_N(q,\eps)
    =
    \lt\{
    (\bsig,\bsig')\in\cS_N^2~:~
    \lt|\frac{\la\bsig,\bsig'\ra}{N}-q\rt|\leq \eps.
    \rt\}.
\]

The following upper bound will be useful in the \revv{sequel}.
\begin{lemma}\label{lemma:RS-Formula}
For any $t\in [0,1]$, we have
\begin{align}
F_{\beta}(\xi)\le \Par_{\beta,\sRS}(t;\xi):= \frac{\beta^2}{2}
\big[\xi(1)-\xi(t)\big]+\frac{1}{2}\frac{t}{1-t}+\frac{1}{2}\log(1-t)\,.
\end{align}
\end{lemma}
\begin{proof}
Substitute $\zeta=\delta_t$ in Eq.~\eqref{eq:GeneralVariational}.
\end{proof}

\subsection{Planted model and contiguity}

\revv{
Below we define a ``planted'' model (i.e.\ distribution) $\nu_{\pl}$ on pairs $(\bx,\bG) \in \cS_N \times (\bbR^{N})^{\otimes p}$ and show that for $\beta\leq \beta_c(p)$, it enjoys contiguity at exponential scale with the original ``random'' model $\nu_{\rd}$ in which $\bx$ is drawn from the spherical spin glass Gibbs measure corresponding to disorder $\bG$.
First, the law of the planted model is given by:
}
\begin{equation}
    \nu_{\pl}(\de \bx, \de \bG) 
    = 
    \frac{1}{Z_{\pl}}\, 
    \exp\Big(-\frac{1}{2}
    \lt\|\bG - \frac{\beta \bx^{\otimes p}}{N^{\frac{p-1}{2}}} \rt\|_{F}^2 \, 
    \Big)  \, \mu_0(\de \bx) \, \de \bG \, ,
\end{equation}
where $\de \bG$ is Lebesgue measure on $ (\bbR^{N})^{\otimes p}$. Notice that the normalizing constant
\begin{align}\label{eq:Zpl}
    Z_{\pl} := \int\! \exp\Big(- \frac{1}{2}
    \lt\|\bG - \frac{\beta \bx^{\otimes p}}{N^{\frac{p-1}{2}}} \rt\|_{F}^2 \, 
    \Big)\, \de \bG 
\end{align}
is independent of $\bx \in \cS_N$. 
It is easy to see by symmetry that the marginal distribution of $\bx$ under $\nu_{\pl}$ is $\mu_0$. Meanwhile, under the conditional law $\nu_{\pl}( \, \cdot \, | \bx)$ the tensor $\bG$ has a rank-one spike $\beta \bx^{\otimes p}/N^{\frac{p-1}{2}}$. Namely, under $\nu_{\pl}( \, \cdot \, | \bx)$, we have
 \begin{align}
     \bG = \frac{\beta}{N^{\frac{p-1}{2}}}\bx^{\otimes p}
     +
     \bW\, ,\label{eq:PlantedDef}
 \end{align}
 for $\bW$ a tensor with i.i.d.\ standard normal entries independent of
 $\bx$.
Meanwhile, the conditional law $\nu_{\pl}( \, \cdot \, | \bG)$ of $\bx$ given $\bG$ is the Gibbs measure $\mu_{\beta}$ (since $\|\bx\|_2^2$ is constant for $\bx\in \cS_N$).  

Let $\nu_{\rd}$ be the distribution  of $(\bx,\bG) \in \cS_N \times (\bbR^{N})^{\otimes p}$, such that
$\bG\in (\reals^{N})^{\otimes p}$ has i.i.d. standard normal entries
and, conditionally on $\bG$, $\nu_{\rd}(\,\cdot\, |\bG)=\mu_{\beta}(\,\cdot\,)$.
Namely, this is the joint distribution of the disorder under our original
model and a configuration from the corresponding Gibbs measure.
Then we have
\[
    \nu_{\pl}(\de\bsig,\de \bG) 
    =
    \nu_{\rd}(\de\bsig,\de \bG)\, \wtZ_{\beta}(\bG)
\]
where $\wtZ_{\beta}(\bG)$ is the rescaled partition function for the pure $p$-spin model,
\begin{equation}
    \wtZ_{\beta}(\bG)
    :=
    \frac{\int
    e^{\beta H_N(\bx)}\de \mu_0(\bx)}
    {\bbE\int
    e^{\beta H_N(\bx)}\de \mu_0(\bx)}
    \, .
\end{equation}

To transfer results on the planted model, we require a contiguity result, or equivalently control on the fluctuations of the 
partition function $\wtZ_{\beta}(\bG)$.
A sharp result of this type was proven in 
\cite{jagannath2020statistical} for $p\ge 6$ even.  
\begin{theorem}[Theorem 1.1 in \cite{jagannath2020statistical}]
Suppose $p\ge 6$ is even and $\beta<\beta_c(p)$. Then we have
\begin{align}
\lim_{N\to\infty}\big\|\nu_{\pl}-\nu_{\rd}\big\|_{\sTV} = 0\, .
\end{align}
\end{theorem}
\begin{remark}
For all other values of $p$   \cite[Theorem 2]{montanari2015limitation}
yields the same result via a simple second moment argument. However, this requires
$\beta<\overline{\beta}(p)$, where $\overline{\beta}(p)$ is the point at which the 
second moment of $Z_N$ becomes exponentially larger than the square of the first moment.
This condition is not sharp.

We point out that similar results are available for other prior
distributions of $\bx$ (different from the uniform distribution over the sphere) from
\cite{chen2019phase,chen2021phase}.
\end{remark}

To the best of our knowledge, no sharp result is 
available for $p$ odd (the condition $\beta<\overline{\beta}(p)$ being stronger
than needed)\footnote{Subsequent to the initial posting of this paper, \cite[Section 8]{huang2024sampling} established contiguity (with log-normal fluctuations) for all $\beta<\beta_c(p)$.}. For our purposes, however, a weaker notion of
contiguity is already sufficient.
Given a sequence of events $(E_N)_{N\geq 1}$ in a sequence of probability spaces, we say that $E_N$ is \emph{exponentially likely} \revv{(at speed $N$)} if $\bbP(E_N)\geq 1-c^{-1}e^{-cN}$ for some $c>0$.
\begin{lemma}
\label{lem:exp-contiguity-abstract}
   For $p\ge 3$, $\beta\leq \beta_c(p)$ and $\eps>0$, let $\bG \sim \nu_{\rd}$. Then the event 
    \[
    \lt|\frac{1}{N}\log \wtZ_{\beta}(\bG)\rt|\leq \eps
    \]
    is exponentially likely.
\end{lemma}
\begin{proof}
    This is immediate from the equality in \eqref{eq:free-energy-def} and the exponential
     concentration of free energy (the latter is a Lipschitz function of the Gaussian disorder). 
\end{proof}

The next lemma allows us to transfer exponentially high probability results from $\nu_{\pl}$ to $\nu_{\rd}$.
\begin{lemma}
\label{lem:exp-contiguity-concrete}
    If the sequence $(E_N)_{N\geq 1}$ is exponentially likely under $\nu_{\pl}$, then it is also exponentially likely under $\nu_{\rd}$.
\end{lemma}
\begin{proof}
    We work with the obvious complementary notion of exponential unlikeliness: $\nu_{\pl}(E^c_N)\leq c^{-1}e^{-cN}$. We  have
    \begin{align*}
    \nu_{\rd}(E^c_N)
    &=
    \nu_{\pl}\big(\wtZ_{\beta}^{-1}\cdot 1_{E^c_N}\big)
    \\
    &\leq 
    e^{\eps N}
    \nu_{\pl}\big(E_N^c\big)
    +
    \nu_{\pl}\big(\wtZ_{\beta}^{-1}\cdot 1_{\wtZ_{\beta}^{-1}\geq e^{\eps N}}1_{E^c_N}\big)
    \\
    &=
    e^{\eps N}
    \nu_{\pl}\big(E^c_N\big)
    +
    \nu_{\rd}\big(1_{\wtZ_{\beta}^{-1}\geq e^{\eps N}}1_{E^c_N}\big)
    \\
    &\leq 
    c^{-1} e^{(\eps-c) N}
    +
    \nu_{\rd}\big(1_{\wtZ_{\beta}^{-1}\geq e^{\eps N}}\big)
    .
    \end{align*}
    Taking $\eps=c/2$ and applying Lemma~\ref{lem:exp-contiguity-abstract} to the last term completes the proof.
\end{proof}

\subsection{Non-monotonicity of Franz--Parisi potential}

For $\bsig\sim\mu_{\beta}$ a typical Gibbs sample, we consider the Franz--Parisi potential
\begin{equation}\label{eq:fp}
    \cF_{\beta}(q) :=
    \lim_{\eps\downarrow 0}
    \plim_{N\to\infty} 
    \frac{1}{N}
    \log 
    \int e^{\beta H_N(\bsig')}
    1_{(\bsig,\bsig')\in \cS_N(q,\eps)}
    \de\mu_0(\bsig')\, .
\end{equation}
The proofs presented below will imply existence of these limits in the regime of interest.
\revv{
Our strategy will be to compute \eqref{eq:fp} in the planted model where $\bsig$ is drawn uniformly from the sphere. Exponential concentration of the involved free energies lets us transfer the obtained formula to the neighborhood of a typical Gibbs sample. 
}

To establish shattering, we would like to show that $\cF_{\beta}(q)$ increases for $q$ inside an interval $I_p=[\lbq_p,\ubq_p]$.
This will allow us to construct a shattered decomposition of $\mu_{\beta}$ by
constructing balls of suitable radius around typical Gibbs samples. 
To guarantee disjointness of the resulting clusters, we will need to ensure that the ratio $(1-\lbq_p)/(1-\ubq_p)$ is large.
(Section \ref{sec:Improved} removes the constraint that this ratio is large,
at the price of a more elaborate construction of clusters.)

In this subsection we prove the following result ensuring such behavior for suitable $(\beta,p)$.

\begin{proposition}
\label{prop:Franz-Parisi-deriv-positive}
    There exists an absolute constant $C>0$ such that the following holds. 
    For $p\ge 3$ and $\beta<\beta_c(p)$, we have
    \[
    \frac{\de}{\de q} \cF_{\beta}(q)>0\, ,
    \quad\quad
    \forall~
    q\in \lt[1-\frac{1}{2p},1-\frac{C}{\beta p}\rt].
    \]
\end{proposition}

Before proving this estimate, we find an exact formula for $\cF_{\beta}(q)$ via the planted model. Let
\begin{equation}
\label{eq:xi-q-def}
    \xi_q(x)
    =
    (q^2+(1-q^2)x)^p-q^{2p}\, .
\end{equation}
and recall that $F_{\beta}(\xi_q)$ is the single-replica quenched free energy of the
 mixed $p$-spin model with covariance given by $\xi_q$, cf. Eq.~\eqref{eq:_mixed}.
We denote the corresponding Hamiltonian $H_{N}^{(\xi_q)}$
(this is a centered Gaussian process with covariance as in Eq.~\eqref{eq:Hxi}).
\begin{proposition}
\label{prop:band-model}
    Assume $(\bsig,\bG)\sim\nu_{\pl}$ are distributed according
    to the planted model and let $H_N$ be the Hamiltonian associated to $\bG$.
    For 
    $\bsig'\in \revv{\cS_{N}\cap \bsig^{\perp}}$
    , and $\bU\in\reals^{N\times (N-1)}$ 
    an orthonormal basis of $\bsig^{\perp}$, define
    \begin{align}
    \wt H_N(\bsig')&:=H_N(q\bsig+\sqrt{1-q^2} \bU\bsig')-H_N(q\bsig)\, ,
    \end{align}
    Then the collection $\{\wt H_N(\bsig'):\bsig'\in\revv{\cS_{N}\cap \bsig^{\perp}}\}$ is independent of $H_N(q\bsig)$ and has the same law as $\{H_{N}^{(\xi_q)}(\bsig') : \bsig'\in\revv{\cS_{N}\cap \bsig^{\perp}}\}$.
\end{proposition}

\begin{proof}
Let $f_q(\bsig'):=q\bsig+\sqrt{1-q^2} \bU\bsig'$.
    It is clear that $\wt H_N$ is a centered Gaussian process, and independence with $H_N(q\bsig)$ easily follows from the covariance formula of $H_N$. Its covariance is  
    \[
    \bbE[\wt H_N(\bsig')\wt H_N(\bsig'')]
    =
    \bbE\lt[
    \Big(H_N(f_q(\bsig'))-H_N(q\bsig)\Big)
    \Big(H_N(f_q(\bsig''))-H_N(q\bsig)\Big)
    \rt].
    \]
    Note that the subtraction defining $\wt H_N$ cancels all spike contributions, so we can evaluate the latter expectation in the random model (with no special conditioning on $\bsig$) instead of the planted one:
    \begin{align*}
    \frac{1}{N}\bbE \big[\wt H_N(\bsig')\wt H_N(\bsig'')\big]
    &= \Big(\frac{\la f_q(\bsig'),f_q(\bsig'')\ra}{N}\Big)^p-
    \Big(\frac{\la f_q(\bsig'),q\bsig\ra}{N}\Big)^p -
    \Big(\frac{\la f_q(\bsig''),q\bsig\ra}{N}\Big)^p+q^{2p}\\
    &= \Big(q^2+(1-q^2)\frac{\la \bsig',\bsig''\ra}{N}\Big)^p - q^{2p} \, .
    \end{align*}
    This proves the claim.
\end{proof}

Next we relate the Franz--Parisi potential~\eqref{eq:fp} to $F_{\beta}(\xi_q)$:
\begin{proposition}
\label{prop:Franz-Parisi-formula}
    For $\beta\leq \beta_c$,
    \[
    \cF_{\beta}(q)=
    F_{\beta}(\xi_q)
    +\beta^2 q^p+\frac{1}{2}\log(1-q^2)\, .
    \]
    Moreover exponential concentration holds, i.e.\ for any $\eta>0$ \revv{ and $q\in (-1,1)$,
    there exists $\eps_0=\eps_0(\eta,p,\beta,q)$} such that, 
    for $\eps\in (0,\eps_0)$, the event 
    \[
    \lt|
    \frac{1}{N}
    \log 
    \int e^{\beta H_N(\bsig')}
    1_{(\bsig,\bsig')\in \cS_N(q,\eps)}
    \de\mu_0(\bsig')
    -\cF_{\beta}(q)
    \rt|
    \leq \eta
    \]
    is exponentially likely.
\end{proposition}

\begin{proof}
Define 
\begin{align}
\cF^{\rd}_{N,\beta}(q,\eps):=\frac{1}{N}
    \log 
    \int e^{\beta H_N(\bsig')}
    1_{(\bsig,\bsig')\in \cS_N(q,\eps)}
    \de\mu_0(\bsig')\, ,\label{eq:Frd}
\end{align}
where the superscript indicates that $\cF^{\rd}_{N,\beta}(q,\eps)$ is a function of $(\bsig,\bG)\sim\nu_{\rd}$. We denote by $\cF^{\pl}_{N,\beta}(q,\eps)$
the corresponding random variable for $(\bsig,\bG)\sim\nu_{\pl}$.

Recall that, under $\nu_{\pl}$, $\bG$ is given by
 \begin{align}
     \bG = \frac{\beta}{N^{\frac{p-1}{2}}}\bsig^{\otimes p}
     +
     \bW\, ,\label{eq:PlantedDef2}
 \end{align}
 where $\bW$ is a Gaussian tensor independent of $\bsig$.
\revv{ By rotational invariance, we may assume that $\bsig$ is fixed in \eqref{eq:PlantedDef2}.  
Using Lipschitz concentration for $\cF^{\pl}_{N,\beta}(q,\eps)$ in the Gaussian random variables $\bW$ yields:  
}
%
\begin{align} 
\prob\Big(\big|\cF^{\pl}_{N,\beta}(q,\eps)-\E\cF^{\pl}_{N,\beta}(q,\eps)\big|\ge t\Big)
\le 2\, e^{-c_0 Nt^2}\, ,
\end{align}
for some $c_0=c_0(\beta,p,q,\eps)$.

Hence, by Lemma~\ref{lem:exp-contiguity-concrete} it is sufficient 
to control the typical value of $\cF^{\pl}_{N,\beta}(q,\eps)$. Note that 
\begin{align*}
\cF^{\pl}_{N,\beta}(q,\eps) =
    \frac{1}{N}\beta H_N(q\bsig)+ \frac{1}{N}  \log 
    \int e^{\beta [H_N(\bsig')-H_N(q\bsig)]}
    1_{(\bsig,\bsig')\in \cS_N(q,\eps)}
    \de\mu_0(\bsig')\, .
\end{align*}
By Eq.~\eqref{eq:PlantedDef2}, we have
\begin{align*}
\frac{1}{N}   \beta H_N(q\bsig) = \beta^2q^p + 
\frac{\beta} {N^{(p+1)/2}} \la\bW,\bsigma^{\otimes p}\ra\, .
\end{align*}
The second term is normal with mean zero and variance $\beta^2/N$,
and hence we can safely neglect it.
Recall the definition of the effective mixture function $\xi_q$ from \eqref{eq:xi-q-def}. 
Using Proposition~\ref{prop:band-model} and Lipschitz continuity of $H_N$ \revv{(i.e. Proposition~\ref{prop:norm-bound} with $j=1$)}, we get
\begin{align}
\label{eq:eps-error-bound}
\cF^{\pl}_{N,\beta}(q,\eps) \stackrel{d}{=}
     \beta^2q^p   + \frac{1}{2}  \log (1-q^2)+\frac{1}{N}\log
    \int_{\revv{\cS_{N}\cap \bsig^{\perp}}} e^{\beta H^{(\xi_q)}_N(\bsig')}
    \de\mu_0^{N-1}(\bsig') +O_{\bbP}(\eps)\, ,
\end{align}
where the term $\frac{1}{2}\log(1-q^2)$ accounts for the decreased volume of the band and $\mu_0^{N-1}$ denotes the uniform measure on the subsphere $\revv{\cS_{N}\cap \bsig^{\perp}}$.
Finally, the desired result follows since the expectation of the third term above is $F_{\beta}(\xi_q)$ in the limit $N \to \infty$.
\revv{Indeed, up to the (negligible) rescaling of the coefficients of $H^{(\xi_q)}_N$ by $(\frac{N}{N-1})^{(p-1)/2}$, it is exactly the free energy of a spherical spin glass on $\cS_{N-1}$.}
\end{proof}

Having established in particular that the limit $\cF_{\beta}(q)$ exists, we stop to record the following simple estimate. 

\begin{proposition}
\label{prop:Jensen-bound}
For any $\beta<\beta_c(p)$ of Eq.~\eqref{eq:PspinExplicit}, and all $q\in [0,1]$, we have
\begin{align}
 \cF_{\beta}(q)&\leq \cF_{\beta}^{\sRS}(q)\, ,\label{eq:RS-Franz-Parisi}
 \\
 \cF_{\beta}^{\sRS}(q)& := \frac{\beta^2}{2}\big(1+q^p\big)+\frac{q}{2}+\frac{1}{2}\log(1-q)\, .
 \end{align}
\end{proposition}
\begin{proof}
Using Proposition \ref{prop:Franz-Parisi-formula} and Lemma \ref{lemma:RS-Formula},
we get, for any $t\in [0,1)$
\begin{align*}
 \cF_{\beta}(q) & =
    F_{\beta}(\xi_q)
    +\beta^2 q^p+\frac{1}{2}\log(1-q^2)\\
    &\le \Par_{\beta,\sRS}(t;\xi_q)+\beta^2 q^p+\frac{1}{2}\log(1-q^2)\, .
\end{align*}
The claim \eqref{eq:RS-Franz-Parisi} follows by substituting $t=q/(1+q)$, and noting
that
\begin{align*}
\cF_{\beta}^{\sRS}(q)=\Par_{\beta,\sRS}(q/(1+q));\xi_q)+\beta^2 q^p+\frac{1}{2}\log(1-q^2)\, .
\end{align*}
The substitution comes from the heuristic that if $\bsig_1,\bsig_2$ are two replicas such that $\langle \bsig_1,\bsig_2\rangle/N \simeq q$, and  $\bsig_1 = q\bsig + \sqrt{1-q^2}\bs_1$, $\bsig_2 = q\bsig + \sqrt{1-q^2}\bs_2$ where $\bs_1,\bs_2$ are orthogonal to $\bsig$ (the planted vector), then $\langle \bs_1,\bs_2\rangle/N \simeq q/(1+q)$.
\end{proof}

The main importance of Proposition~\ref{prop:Jensen-bound} is actually the following consequence. It will ensure that claim~\ref{it:small-prob} holds in Theorems~\ref{thm:main} and \ref{thm:improved-clustering}, and is also used to deduce Theorem~\ref{thm:disorder-chaos} in the latter setting.

\begin{corollary}
\label{cor:FP-maximized}
    For $\beta<\beta_c(p)$ we have
    \begin{align*}
     \cF_{\beta}(q)\leq F_{\beta}  \;\;\;  \forall q\in (-1,1)\,.
     \end{align*}
     with equality if and only if $q=0$. In particular, $\cF_{\beta}(q)$ is uniquely maximized at $q=0$.
\end{corollary}

\begin{proof}
    \revv{We first apply Proposition~\ref{prop:Jensen-bound}, which gives that for} $\beta<\beta_c(p)$ and $q\in (0,1)$,
\[
\frac{2}{q^p}\big[F_{\beta}-\cF_{\beta}^{\sRS}(q) \big]
\geq 
\frac{2}{q^p}\big[F_{\beta}-\cF_{\beta}(q) \big]
=
-\beta^2+\frac{1}{q^p}\log\left(\frac{1}{1-q}\right)
-\frac{1}{q^{p-1}}
\stackrel{\eqref{eq:PspinExplicit}}{>}
0\, .
\]
It follows by \eqref{eq:xi-q-def} and Proposition~\ref{prop:Franz-Parisi-formula} that $\cF_{\beta}(q)\leq \cF_{\beta}(|q|)$, which completes the proof.
\end{proof}

We now turn to the main proof of this subsection.

\begin{proof}[Proof of Proposition~\ref{prop:Franz-Parisi-deriv-positive}]
Using the result of Proposition~\ref{prop:Franz-Parisi-formula}, the derivative of $\cF_{\beta}(q)$ is
\begin{equation}
\label{eq:deriv-Franz-Parisi-initial}
    \frac{\de}{\de q} \cF_{\beta}(q)
    =
    \frac{\de}{\de q}F_{\beta}(\xi_q)
    +
    \beta^2 p q^{p-1}
    -
    \frac{q}{1-q^2}\, .
\end{equation}
Note that
\[
    \frac{\de}{\de q}\xi_q(x)
    =
    2pq\Big((1-x)\big(q^2+(1-q^2)x\big)^{p-1}-q^{2p-2}\Big)
    \, .
\]

Recall the Crisanti-Sommers formula \eqref{eq:CrisantiSommers} which we rewrite,
integrating by parts,
\begin{align}
\Par_{\beta}(\zeta;\xi) 
&=
\frac{1}{2}\left(\beta^2\bbE_{x\sim \zeta}
    \lt[\xi(1)-\xi(x)\rt]
\int_0^{\hat{q}} \frac{\de t}{\phi_{\zeta}(t)}  +\log(1-\hat{q})\right)\, .
\end{align}
Recall that $(\zeta,q)\mapsto\Par_{\beta}(\zeta;\xi_q)$ is strictly convex
and lower semicontinuous in $\zeta$. Further, by dominated convergence it is easily
seen to be differentiable in $q$ for any fixed $\zeta$.
Finally, it is easy to show that for (say) $q\geq 1/2$ and fixed $p$, the minimization of the Crisanti--Sommers functional \eqref{eq:CrisantiSommers} can be restricted to $\zeta$ with $\hat q\leq \tilde q(p)<1$ bounded away from $1$ depending only on $p$.
This represents $F_{\beta}(\xi_q)$ as the minimum of $\Par_{\beta}(\zeta;\xi_q)$ over a weak*-compact set of $\zeta$, namely cumulative distribution functions of probability measures on $[0,\hat q(p)]$.
Therefore with $\zeta_q=\zeta_{\xi_q}$ the minimizer to \eqref{eq:GeneralVariational} for $\xi_q$, 
the envelope theorem \cite[Theorem 2 and Corollary 4]{milgrom2002envelope} implies
\begin{equation}
\begin{aligned}
\label{eq:F-beta-deriv-LB}
    \frac{\de}{\de q}F_{\beta}(\xi_q)
    &= \left.\frac{\partial\phantom{q}}{\partial q}\Par_{\beta}(\zeta;\xi_q) \right|_{\zeta=\zeta_{q}}\\
    & = 
    \frac{\beta^2}{2}\cdot  \bbE_{x\sim \zeta_{\xi_q}}
    \lt[\frac{\de}{\de q}\big(\xi_q(1)-\xi_q(x)\big)\rt]
    \\
    &= 
    -\beta^2 pq\cdot  \bbE_{x\sim \zeta_{\xi_q}}
    \lt[
    (1-x)\big(q^2+(1-q^2)x\big)^{p-1}
    \rt]
    .
\end{aligned}
\end{equation}
It remains to estimate the expectation in \eqref{eq:F-beta-deriv-LB}.
Taking a derivative with respect to $\beta$ instead of $q$, \cite[Theorem 1.2]{talagrand2006parisimeasures} and convexity of $x\mapsto \xi_q(x)$ now imply
\begin{equation}
\label{eq:temp-derivative-free-energy}
\begin{aligned}
    \frac{\de}{\de\beta}F_{\beta}(\xi_q)
    &=
    \beta \cdot
    \bbE_{x\sim \zeta_{\xi_q}}\big[\xi_q(1)-\xi_q(x)\big]
    \\
    &\geq 
    \beta \cdot
    \bbE_{x\sim \zeta_{\xi_q}}\lt[(1-x)\xi_q'(x)\rt]
    \\
    &=
    \beta p(1-q^2) \cdot
    \bbE_{x\sim \zeta_{\xi_q}}
    \lt[(1-x)\big(q^2+(1-q^2)x\big)^{p-1}\rt]
    \, .
\end{aligned}
\end{equation}
Next we claim that for $q\in [1-(2p)^{-1},1)$,
\begin{equation}\label{eq:bound_beta_prime}
  \frac{\de}{\de\beta} F_{\beta}(\xi_q)\leq C\, 
\end{equation}
for some absolute constant $C>0$.
Let us first derive our conclusion assuming the above claim. Plugging the last two displays into \eqref{eq:F-beta-deriv-LB}, we obtain
\[
    \frac{\de}{\de q}F_{\beta}(\xi_q)
    \geq 
    -\frac{C\beta q}{1-q^2} \, .
\]
For $ 1-(2p)^{-1}< q \le 1 \le C\beta$, we conclude from \eqref{eq:deriv-Franz-Parisi-initial} that
\begin{equation}
\label{eq:good-bound-Franz-Parisi}
    \frac{\de}{\de q} \cF_{\beta}(q)
    \geq 
    -\frac{(C\beta+1)q}{1-q^2}
    +
    \beta^2 p q^{p-1}
    > 
    -\frac{C\beta}{1-q}
    +
    \frac{\beta^2 p}{2}\, .    
\end{equation}
By inspection, the right-most expression in \eqref{eq:good-bound-Franz-Parisi} is positive whenever
    \[
    1-\frac{1}{2p}\leq q\leq 1-\frac{2C}{\beta p}\, , ~~~~ \max \{C^{-1}, 4C\} \le \beta\, .
    \]

We now prove the claim~\eqref{eq:bound_beta_prime}.
Let $H_{N}^{(\xi_q)} = \sum_{k=1}^p\sqrt{\xi_{q,k}}
H_N^{(k)}(\bsig)$ be the Hamiltonian corresponding to the mixture
$\xi_q$, where the $H_N^{(k)}$'s are independent $k$-spin Hamiltonians,
and $\xi_{q,k}$ is the coefficient of $x^k$ in $\xi_q(x)$.
In the following, we restrict ourselves to the overlap range $q\in [1-(2p)^{-1},1)$. Since  $\binom{p}{j}\leq p^j/j!$, in this case we get 
$\xi_{q,k}\le 1/k!$. 
Proposition~\ref{prop:norm-bound} now yields
\begin{align*}
    \limsup_{N\to\infty}\E
    \sup_{\bsig\in\cS_N}
    \frac{1}{N} H_{N}^{(\xi_q)}(\bsig)
    &\le  \limsup_{N\to\infty}\,
   \sum_{k=1}^p\sqrt{\xi_{q,k}} \, \E \,\sup_{\bsig\in\cS_N}\frac{1}{N} H_N^{(k)}(\bsig)\\
   &\le \frac{C}{10}
   \sum_{j=1}^p
    \sqrt{\xi_{q,j}} 
    \sqrt{\log j}\\
   & \leq \frac{C}{10}   \sum_{j=1}^p\sqrt{\frac{1}{j!}\log j} \le C\, .
\end{align*}
Finally, to obtain~\eqref{eq:bound_beta_prime} we bound the average energy of the model by its maximum: denoting by $\langle \,\cdot \,\rangle_{(\xi_q)}$ the average w.r.t.\ the Gibbs measure on $\cS_N$ with Hamiltonian $H^{(\xi_q)}$, we have
\[  \frac{\de}{\de\beta} F_{\beta}(\xi_q)\le \limsup_{N\to \infty} \frac{1}{N} \E\left\langle  H^{(\xi_q)}(\bsig) \right\rangle_{(\xi_q)} \le \limsup_{N\to \infty}\frac{1}{N} \E \sup_{\bsig \in \cS_N} H_{N}^{(\xi_q)}(\bsig) \, , \]
where the first inequality follows by convexity of the free energy w.r.t.\ $\beta$. 
\end{proof}

\section{Construction of the shattered decomposition}

\revv{
In this section we complete the proof of Theorem~\ref{thm:main} by constructing a shattering decomposition for pure $p$-spin models with large $p$. 
The idea is that, because $\cF_{\beta}$ is increasing on a somewhat large interval near $1$, any pair of clusters which are ``faithfully described'' by $\cF_{\beta}$ are deterministically well-separated. 
In fact in the range of $(p,\beta)$ under consideration, this separation distance is larger than the radius of a typical cluster.
This allows us to define an equivalence relation of nearby typical points; any representative from each class then suffices as the center of a spherical cap cluster.
}

Throughout this proof we assume $\beta\in (\ubeta(p),\beta_c(p))$
where $\ubeta(p)\ge 4\max\{C^{-1},C\}$ is a sufficiently large absolute constant, and $C$ is the constant in 
Proposition~\ref{prop:Franz-Parisi-deriv-positive}.
 We define
\[
    \lbq_p
    =
    1-\frac{1}{2p}\, , 
    \quad\quad\quad\quad
    \ubq_p
    =
    1-\frac{C}{\beta p}\revv{>\lbq_p} \, ,
\]
and the corresponding inner and outer radii
\[
    r(\beta,p)
    =
    \sqrt{2(1-\ubq_p)}\asymp 1/\sqrt{\beta p}\, ,
    \quad\quad\quad\quad
    R(\beta,p)
    =
    \sqrt{2(1-\lbq_p)}\asymp 1/\sqrt{p}\, .
\]
We will often omit the arguments $\beta$, $p$.
Note that $R\geq 1000\, r$ for $\ubeta(p)$ sufficiently large, which we assume throughout the proof.
Further, recall the notation $\|\bx\|_N:=\|\bx\|/\sqrt{N}$ for $\bx\in\reals^N$.

We now define a set of \emph{regular} points which will play an important role in the construction of the shattering decomposition. 
\begin{definition}
\label{def:Sreg}
    Let $S_{\sreg}(H_N,\beta,p,c)\subseteq\cS_N$ consist of all $\bsig\in\cS_N$ such that 
    \begin{align}
    \label{eq:Sreg-1}
    \int e^{\beta H_N(\bsig')}
    1_{\|\bsig-\bsig'\|_N \in [r,R]}
    \de\mu_0(\bsig')
    &\leq 
    e^{-cN}
     \int e^{\beta H_N(\bsig')}
    1_{\|\bsig-\bsig'\|_N\leq r}
    \de\mu_0(\bsig'),
    \\
    \label{eq:Sreg-2}
     \int e^{\beta H_N(\bsig')}
    1_{\|\bsig-\bsig'\|_N\leq 3r}
    \de\mu_0(\bsig')
    &\leq 
    e^{-cN}
    \int e^{\beta H_N(\bsig')}
    \de\mu_0(\bsig')\,.
    \end{align}
\end{definition}

\begin{proposition}
\label{prop:S-reg-covers}
    For $p$ large enough and $\beta\in (\ubeta(p),\beta_c(p))$, $c>0$ small enough, 
    there exists $c'(\beta,p,c)>0$ such that
    \begin{equation}
    \label{eq:S-reg-big}
    \bbE\big[\mu_{\beta}\big(S_{\sreg}(H_N,\beta,p,c)\big)\big]\geq 1-e^{-c'N}.
\end{equation}
\end{proposition}
\begin{proof}
\revv{
    Using contiguity at exponential scale, it suffices to show $\bbP(\bsig\in S_{\sreg}(H_N,\beta,p,c))\geq 1-e^{-c'N}$ holds under the planted model.
    The first condition \eqref{eq:Sreg-1} for $S_{\sreg}$ follows readily from the non-monotonicity of $\cF_{\beta}$.
    The main point is that by Proposition~\ref{prop:Franz-Parisi-deriv-positive}, $q \mapsto \cF_{\beta}(q)$ is strictly increasing over the interval $(\lbq_p,\overline{\overline{q}}_p)$ with $\overline{\overline{q}}_p = 1-\frac{C}{\beta p} > \ubq_p$.
    To be precise, choose an $\eps$-net $\cN=\{q_1,\dots,q_k\}\subseteq [\lbq_p,\overline{\overline{q}}_p]$.
    Combining Proposition~\ref{prop:Franz-Parisi-formula} for each $q_j$ with the Lipschitz bound \eqref{eq:C-bounded} (or directly using \eqref{eq:eps-error-bound} for each $q_j$) then establishes the first condition for $S_{\sreg}$ with probability $1-e^{-c'N}$ under the planted model.}

\revv{    For the second condition \eqref{eq:Sreg-2}, we similarly discretize $(\sqrt{1-9r^2},1)$ to an $\eps$-net and apply Corollary~\ref{cor:FP-maximized} (as well as concentration of the relevant band-restricted free energy) to each discretization point.
    Note that since $\cF_{\beta}(q)$ is continuous and tends to $-\infty$ as $q\uparrow 1$, Corollary~\ref{cor:FP-maximized} does imply the uniform bound $\sup_{q\in (\sqrt{1-9r^2},1)}\cF_{\beta}(q)<F_{\beta}$. 
    Given this, we can choose $\eps$ small depending on the margin of the latter inequality.
}
\end{proof}

We next observe that $S_{\sreg}$ deterministically obeys a clustering property.
\begin{lemma}
\label{lem:OGP}
    If $\bsig,\brho\in S_{\sreg}$ then $\|\bsig-\brho\|_N
    \not\in [3r,R/3]$.
\end{lemma}

\begin{proof}
    Suppose $\|\bsig-\brho\|_N \in [3r,R/3]$.
    Then for any $\bx\in\cS_N$, we have
    \[
    1_{\|\bsig-\bx\|_N\in [r,R]}
    +
    1_{\|\brho-\bx\|_N\in [r,R]}
    \geq 
    1_{\|\bsig-\bx\|_N\leq r}
    +
    1_{\|\brho-\bx\|_N\leq r}
    .
    \]
    Weighting by $e^{\beta H_N(\bx)}$ and integrating contradicts the definition of $S_{\sreg}$.
\end{proof}

We now construct a shattering decomposition as follows. For each $\bsig\in S_{\sreg}$, consider the cluster 
\begin{equation}
    \cC_{\bsig}^{\circ}=S_{\sreg}\cap \Ball_{3r}(\bsig)\, .
\end{equation}

\begin{proposition}
\label{prop:clusters}
    For any $\bsig,\brho\in S_{\sreg}$, the clusters $\cC_{\bsig}^{\circ},\cC_{\brho}^{\circ}$ are either identical or $R\sqrt{N}/4$-separated.
\end{proposition}

\begin{proof}
    First suppose $\|\bsig-\brho\|_N\leq 3r$. Then Lemma~\ref{lem:OGP} implies 
\begin{align*}
    \cC_{\bsig}^{\circ}
    =
    S_{\sreg}\cap \Ball_{3r}(\bsig)
    =
    S_{\sreg}\cap \Ball_{10r}(\bsig)
    \supseteq 
    S_{\sreg}\cap \Ball_{3r}(\brho)
    =
    \cC_{\brho}^{\circ}\, .
\end{align*}
By symmetry the reverse containment holds as well, hence $\cC_{\bsig}^{\circ}=\cC_{\brho}^{\circ}$.
If $\|\bsig-\brho\|_N\geq R/3$, the balls $\Ball_{3r}(\bsig), \Ball_{3r}(\brho)$ are $(R/4)\sqrt{N}$-separated. 
Hence so are $\cC_{\bsig}^{\circ},\cC_{\brho}^{\circ}$.
\end{proof}

\begin{proof}[Proof of Theorem~\ref{thm:main}]
    For each distinct cluster above, we choose a representative $\bsig\in S_{\sreg}$ such that $\cC_{\bsig}^{\circ}$ is that cluster. 
    Then we take the spherical cap $\cC_{\bsig}=\Ball_{3r}(\bsig)$ as one of the clusters in our set.
    We claim this satisfies the desired properties, with $\Delta_1=3r$ and $\Delta_2=R/10$.
 
    Assertion~\ref{it:small-prob} follows by \revv{\eqref{eq:Sreg-2} in Definition~\ref{def:Sreg}}.
    Assertion~\ref{it:small-diam} follows by construction.
    Assertion~\ref{it:clusters-isolated} holds because each cluster is centered at some $\bsig\in S_{\sreg}$.
    Assertion~\ref{it:clusters-separate} follows by Proposition~\ref{prop:clusters}.
    Assertion~\ref{it:clusters-cover} follows by \eqref{eq:S-reg-big}.
\end{proof}

\section{Transport disorder chaos}
\label{sec:disorder}

In this section we prove Theorem~\ref{thm:disorder-chaos} and its converse, 
Proposition~\ref{prop:chaos-converse}.
The former shows that shattering implies chaos for $\eps_N\geq \Omega(1/N)$, while the latter shows chaos never occurs for smaller $\eps_N$. 
In fact we deduce Theorem~\ref{thm:disorder-chaos} from the following more general statement, which also yields transport disorder chaos for Ising spin glasses from the results of \cite{gamarnik2023shattering}.
We note that condition \ref{it:3r-sep} is implied by \ref{it:clusters-separate} in the setting of Theorem~\ref{thm:main}; it is needed for the more general Theorem~\ref{thm:improved-clustering}, where the \emph{worst-case} separation $s$ between clusters can be quite small.

\begin{theorem}
\label{thm:disorder-chaos-general}
Suppose that $\sup_{\bsig\in\cS_N} 
    |H_N(\bsig)|/N\leq C$, let $\mu_0\in\cP(\cS_N)$ be an arbitrary reference probability measure on the sphere. Suppose that $\mu_{\beta}(\de \bsig)\propto e^{\beta H_N(\bsig)}\mu_0(\de\bsig)$ admits a clustering $\{\cC_1,\dots,\cC_M\}$ of disjoint sets $\cC_m\subseteq\cS_N$, with origin symmetry when $p$ is even, obeying conditions \ref{it:small-prob}, \ref{it:small-diam}, \ref{it:clusters-separate}, \ref{it:clusters-cover} (with constants $r,s,c>0$). Moreover assume $r\leq 0.07$ and:
\begin{enumerate}[label={\sf SEP},ref={\sf SEP}]
    \item 
    \label{it:3r-sep}
    Most cluster pairs are $3r$-separated. Namely,
    defining for $\Delta>0$ the condition
    \[
    \sep_{m_1,m_2,\Delta}
    =
    \lt\{
    d\big(\cC_{m_1},\cC_{m_2}\big)\geq \Delta \sqrt{N}
    \rt\},
    \]
    we assume that (for $r$ as in \ref{it:small-diam}):
    \[
    \sum_{1\leq m_1<m_2\leq M}
    \mu_{\beta}(\cC_{m_1})
    \mu_{\beta}(\cC_{m_2})
    (1-1_{\sep_{m_1,m_2,3r}})
    \leq 0.01.
    \]
\end{enumerate}
If \revv{$\eps_N\geq t/N$}, then there exists a positive constant
$c_*=c_*(p,\beta,C,r,s\revv{,t})>0$ depending only on $(p,\beta,C,r,s\revv{,t})$,
such that,  
\[
\bbE[W_{2,N}(\mu_{\beta},\mu_{\beta}^{\eps_N})~|~\bG]\ge c_*\, .
\]
on the event defined by the above conditions.

In particular, if $\mu_0$ is such that the above conditions hold with uniformly positive probability over $\bG$, then transport disorder chaos \eqref{eq:def-disorder-chaos} follows.
\end{theorem}

\begin{proof}[Proof of Theorem~\ref{thm:disorder-chaos} using Theorem~\ref{thm:disorder-chaos-general}]
    By the law of total expectation, it remains to show the conditions of Theorem~\ref{thm:disorder-chaos-general} are implied (with high probability) by those of Theorem~\ref{thm:main} and Theorem~\ref{thm:improved-clustering}.
    Note that the boundedness of $|H_N(\bsig)|/N$ holds with exponentially good probability by Proposition \ref{prop:norm-bound}.
    In the former case, \ref{it:3r-sep} follows by construction since $\Delta_2\geq \Omega(\sqrt{\beta}\Delta_1)$ holds therein, and $r\leq 0.07$ for $p$ large enough by Remark~\ref{rem:Delta-shrinks-with-p}.

    For the latter case, note that in Theorem~\ref{thm:improved-clustering} the diameter of each cluster is at most $r=7\sqrt{1-\bar q}\leq 0.07$. 
    Since $3r+r+r\leq 0.35$, using the triangle inequality in the first step yields
    \[
    \sum_{1\leq m_1<m_2\leq M}
    \mu_{\beta}(\cC_{m_1})
    \mu_{\beta}(\cC_{m_2})
    (1-1_{\sep_{m_1,m_2,3r}})
    \leq 
    \mu_{\beta}^{\otimes 2}(\|\bsig-\bsig'\|_N\leq 0.35)\leq e^{-c'N}.
    \]
    The latter inequality holds by Corollary~\ref{cor:FP-maximized}, and together with \ref{it:clusters-cover} completes the proof.
\end{proof}

\begin{remark}
    \cite[Theorem 2.9]{gamarnik2023shattering} shows that for $\beta\in (\sqrt{\log 2},\sqrt{2\log 2})$ and $p\geq P(\beta)$, with high probability the Ising spin glass with $\mu_0$ uniform on $\{\pm 1\}^N$ exhibits a shattering decomposition $\cC_1,\dots,\cC_M$ with the following properties.
    \ref{it:small-prob}, \ref{it:small-diam}, \ref{it:clusters-separate}, and \ref{it:clusters-cover} all hold. 
    In fact, there exists a constant denoted $\nu_1>0$ therein, which tends to $0$ as $p\to\infty$ (see the end of \cite[Proof of Proposition 3.7]{gamarnik2023shattering}), such that $r=2\sqrt{\nu_1}$ can be taken.
    (The unusual scaling is because \cite{gamarnik2023shattering} works with Hamming distance.)
    Moreover, one can take $s=2\sqrt{\nu_2}$ for $\nu_2=10\nu_1$, hence $s>3r$ ($\nu_2=3\nu_1$ is written therein, but changing $3$ to $10$ makes no difference \cite{kizildag2023reply}).
    %
    Hence the conditions of Theorem~\ref{thm:disorder-chaos-general} apply, showing that transport disorder chaos \eqref{eq:def-disorder-chaos} holds for pure Ising spin glasses in the aforementioned parameter regime.
\end{remark}

\subsection{Proof of transport disorder chaos}

Assume $\eps_N = \Omega(1/N)$. Let us condition on $\bG$ and fix a decomposition as guaranteed by Theorem~\ref{thm:main}. 
Let $\eta_N=\sqrt{2\eps_N-\eps_N^2}=\Omega(N^{-1/2})$ for convenience.
The conditional law of $\bG^{\eps}$ is 
\begin{align}
    \cL\lt(\bG^{\eps}~|~\bG\rt)
    \stackrel{d}{=}
    (1-\eps_N)\bG
    +
    \eta_N \bW\label{eq:Tensor}
\end{align}
for $\bW$ an i.i.d.\ copy of the disorder. 
The general idea for proving transport disorder chaos is that the fresh disorder $\bW$ randomizes the weights of the clusters $\{\cC_m\}_{m=1}^M$ in such a way that $\mu_{\beta}^{\eps}$ cannot be close to \emph{any} measure depending only on $\bG$.

\begin{proposition}
\label{prop:clustering-robust}
Under the conditions of Theorem~\ref{thm:disorder-chaos-general}, there exists $\eps_0=\eps_0(c)>0$ and $c'>0$ depending only on $(p,\beta,C,r,s)$ such that, 
with probability at least $1-\exp(-c'N)$, the following holds for any 
$\eps\in [0,\eps_0)$:
\begin{align}
\mu_{\beta}^{\eps}\left(\bigcup\limits_{m=1}^M\cC_m\right)\ge 1- e^{-cN/2}\, .
\end{align}
In other words, the new Gibbs measure $\mu_{\beta}^{\eps}$ 
concentrates on exactly on the same clustering with slightly worse parameters.
\end{proposition}

\begin{proof}
With probability at least $1-\exp(-2c'N)$, we have
$\|\bW\|_{\op}, \|\bG\|_{\op}\le C(p)\sqrt{n}$, see 
e.g. Proposition \ref{prop:norm-bound}.
On this event, we can choose $\eps_0$ small enough so 
that the Radon-Nikodym derivative between $\mu_{\beta}$ and 
$\mu_{\beta}^{\eps}$ is uniformly between $[e^{-cN/10},e^{cN/10}]$. 
This immediately implies the desired statement.
\end{proof}

Next, we say that a real-valued random variable is $\rho$-\emph{anti-concentrated} if it has a density on $\bbR$ uniformly bounded by $\rho$ from above. 
In the remainder of the subsection, we also define $\wh\mu_{-0}$ to be the probability measure on $\{1,\cdots,M\}$ selecting an index $i$ with probability  
\begin{equation}
\label{eq:wh-mu-def}
\wh\mu_{-0}(i)=\frac{\mu_{\beta}(\cC_i)}{\sum_{i=1}^M\mu_{\beta}(\cC_i)}.
\end{equation}

\begin{proposition}
\label{prop:cluster-pair-chaos}
    For any two distinct clusters $\cC_i,\cC_j$, such that $\cC_i\neq -\cC_j$ if $p$ is even we define the log-ratio 
    \begin{equation}
    L_{ij}=\log\lt(\frac{\mu_{\beta}^{\eps}(\cC_i)}{\mu_{\beta}^{\eps}(\cC_j)}\rt)\,,~~~~ 1 \le i \neq j\le M \, .
    \end{equation}
    Under the conditions of Theorem~\ref{thm:disorder-chaos-general}, if $\sep_{i,j,3r}$ holds, then the conditional law of $L_{ij}$ given $\bG$ is $(c_0\eta_N\sqrt{N}\big)^{-1}$-anti-concentrated for some $c_0 = c_0(p,\beta,C,r,s)>0$.
    In particular, the set $\Omega\subseteq \{1,\dots,M\}^2$ for which said anti-concentration does not hold satisfies $\wh\mu_{-0}^{\otimes 2}(\Omega)\leq 0.1$.
\end{proposition}

\begin{proof}
    The last claim follows from the first by \ref{it:3r-sep} (and \ref{it:small-prob} for diagonal terms, and \ref{it:clusters-cover} for the denominator in \eqref{eq:wh-mu-def}).

    Given $i,j$ such that $\sep_{i,j,3r}$ holds, we will choose a unit vector $\bv \in \bbR^N$ for which the disorder aligned with $\bv$ causes the random variable $L_{ij}$ to fluctuate nontrivially, even after arbitrary conditioning on the rest of the disorder.
    To do this, we let $\overline{\bW}$ be the orthogonal projection of $\bW$ onto the codimension $1$ tensor subspace orthogonal to $\bv^{\otimes p}$.
    This gives a decomposition
    \[
    \bW=\overline{\bW}+Z\bv^{\otimes p} \, ,
    \]
    where $Z,\overline\bW$ are independent, $Z$ a standard Gaussian and $\<\obW,\bv^{\otimes p}\>=0$.
    We fix $\overline\bW$ and consider $L_{ij}$ as a function of $Z$ only. Computing the derivative of $L_{ij}$ w.r.t.\ $Z$,
    \begin{equation}
    \label{eq:derivative_L}
    L_{i,j}'(Z) = \frac{\beta \eta_N}{N^{(p-1)/2}} \left(\int \langle \bv, \bsig \rangle^p \de \mu_{\beta}^{\eps}\big(\bsig \,\big|\, \cC_i\big)-\int  \langle \bv, \bsig \rangle^p  \de \mu_{\beta}^{\eps}\big(\bsig \,\big|\, \cC_j\big) \right)  
    .
    \end{equation}
   (Note that $\mu_{\beta}^{\eps}$ implicitly depends on $Z$.)

    \revv{
    We set $\bv=\frac{\bsig_i-\bsig_j}{\|\bsig_i-\bsig_j\|}$ for arbitrary $\bsig_i\in\cC_i,\bsig_j\in\cC_j$.
    First since $\sep_{i,j,\revv{3}r}$ holds, 
    $\la \bv,\bsig_i-\bsig_j\ra=\|\bsig_i-\bsig_j\|\geq 3r\sqrt{N}$.
    Fixing any other $\wt\bsig_i\in\cC_i$ and $\wt\bsig_j\in\cC_j$, we moreover have 
    \begin{align*}
    \max\big(|\la \bv,\wt\bsig_i-\bsig_i\ra|,
    |\la \bv,\wt\bsig_j-\bsig_j\ra|\big)
    \leq r\sqrt{N} \, .
    \end{align*}
    Combining, the triangle inequality yields 
    \[
    \la \bv,\wt\bsig_i-\wt\bsig_j\ra\geq r\sqrt{N},\quad\forall\, \wt\bsig_i\in\cC_i,\,\wt\bsig_j\in\cC_j \, .
    \]
    Since $p$ is odd, one has $a^p-b^p\geq 2^{1-p}(r\sqrt{N})^p$ whenever $a,b\in\bbR$ satisfy $a-b\geq r\sqrt{N}$.
    Combining proves that \eqref{eq:derivative_L} is $\Omega(r^p \eta_N \sqrt{N})$ uniformly in $Z$, finishing the proof for $p$ odd.
    }
    
    For even $p$ we argue as follows. 
    Let $\bsig_i\in\cC_i,\bsig_j\in\cC_j$ be arbitrary, and assume without loss of generality that $\la \bsig_i,\bsig_j\ra\geq 0$. (Here we use origin symmetry of the clustering, and the fact that $L_{ij}$ is unchanged if $\cC_j$ is changed to $-\cC_j$.)
    By taking $\bv =\bu/\|\bu\|$, $\bu:=\Proj^{\perp}_{\bsig_j}\bsig_i$, we have $\<\bv,\bsig_j\>=0$ while
    \begin{align*}
    \frac{\<\bv,\bsig_i\>}{\sqrt{N}}
    &=
    \frac{\<\bsig_i,\Proj^{\perp}_{\bsig_j}\bsig_i\>}
    {\|\Proj^{\perp}_{\bsig_j}\bsig_i\|\sqrt{N}}
    = \frac{\|\Proj^{\perp}_{\bsig_j}\bsig_i\|^2}{\|\Proj^{\perp}_{\bsig_j}\bsig_i\|\sqrt{N}}\\
    & = \frac{\|\Proj^{\perp}_{\bsig_j}\bsig_i\|}{\sqrt{N}}
    = \sqrt{1-\frac{\<\bsig_i,\bsig_j\>^2}{N}}\\
    & \stackrel{(\dagger)}{\ge} \frac{\|\bsig_1-\bsig_2\|}{\sqrt{2N}} \ge\frac{3r}{\sqrt{2}}.
    \end{align*}
    Here step $(\dagger)$ is equivalent to $1-t^2\geq 1-t$ for $t\in [0,1]$.
    By Assumption~\ref{it:small-diam} it follows that for \textbf{any} $\bx_i,\bx_j$ in the convex hull of $\cC_i,\cC_j$ respectively, 
    \[
    \lt|\frac{\la \bv,\bx_i\ra}{\sqrt{N}}\rt|
    \geq \frac{3r}{\sqrt{2}}-r
    >
    r
    \geq
    \lt|\frac{\la \bv,\bx_j\ra}{\sqrt{N}}\rt|
    .
    \]
    Hence for this choice of $\bv$, the right-hand side of \eqref{eq:derivative_L} is again $\Omega(r^p \eta_N \sqrt{N})$, uniformly in $Z$.
    In either case, the lower bound $L_{ij}'(Z)\geq \Omega(r^p \eta_N \sqrt{N})$ yields the desired anti-concentration.
\end{proof}

We are now in position to prove transport disorder chaos from shattering: 
\begin{proof}[Proof of Theorem~\ref{thm:disorder-chaos-general}]
     For convenience we assume for now that $\sep_{i,j,3r}$ holds for \textbf{all} $1\leq i<j\leq M$, and indicate at the end of the proof the minor changes necessary to remove this.
     We will prove a stronger statement: given any sequence of $\bG$-measurable probability measures $\wt\mu=\wt\mu_N$ on $\cS_N$, if $\eps_N=\Omega(1/N)$, then
    \begin{equation}
    \label{eq:stronger-statement-arbitrary}
    \lim_{N\to\infty}
    \bbE
    \lt[
    W_{2,N}\lt(\wt\mu,\mu_{\beta}^{\eps}\rt)
    \rt]
    = 
    \Omega\Big(s \min\big(1,\eta_N\sqrt{N}\big)\Big)
    .
    \end{equation}
    From now on we write $(\eps,\eta)$ for $(\eps_N,\eta_N)$.
    
    First, is not hard to see by general principles that the infimum
    $\inf_{\wt\mu}\bbE\lt[
    W_{2,N}\lt(\wt\mu,\mu_{\beta}^{\eps}\rt)
    \rt]$ is non-decreasing in $\eps$.
   Indeed, given $\eps\geq \eps'\geq 0$, we can couple $(\bG,\bG^{\eps'},\bG^{\eps})$ by taking i.i.d.\ copies $\bW,\bW'$ of the disorder and, with $\wt\eps=1-\frac{1-\eps}{1-\eps'}=\frac{\eps-\eps'}{1-\eps'}$:
    \begin{align*}
        \bG^{\eps'}&=(1-\eps')\bG+\sqrt{2\eps'-\eps'^2}\bW,
        \\
        \bG^{\eps}&=(1-\wt\eps)\bG^{\eps'}+\sqrt{2\wt\eps-\wt\eps^2}\bW'.
    \end{align*}
    Then the infimum over $\bG$-measurable $\wt\mu$ of $\bbE\lt[W_{2,N}\lt(\wt\mu,\mu_{\beta}^{\eps}\rt)\rt]$ is at least the infimum over $\bG^{\eps'}$-measurable $\wt\mu$ by the Markov property. This is just the original problem with $\eps$ replaced by $\wt\eps$ which can be arbitrary within $[0,\eps]$. This proves the monotonicity.
    Thus recalling the right-hand side of \eqref{eq:stronger-statement-arbitrary}, we will assume below (without loss of generality) that $\eta=O(1/\sqrt{N})$.

    Next, we compress $\wt\mu$ into a probability measure $\wh\mu$ on the discrete set $\{0,1,\dots,M\}$.
    With 
    \[
    \cC_0=\cS_N\setminus \bigcup\limits_{i=1}^M \cC_i,
    \]
    the complement of the $M$ clusters,
    we define $\wh\mu$ to satisfy 
    \[
    \wh\mu(\{i\})=\wt\mu(\cC_i)\, ,\quad\forall~
    0\leq i\leq M\, .
    \]
    We similarly associate to $\mu^{\eps}_{\beta}$ a probability measure $\wh\mu'$ supported on $\{0,1,\dots,M\}$ such that $\wh\mu'(\{i\})=\mu_{\beta}^{\eps}(\cC_i)$. With $\|\cdot\|_{\sTV}$ denoting the total variation norm, it follows from Assumption~\ref{it:clusters-separate} that
    \begin{equation}
    \label{eq:W2-to-TV}
    W_{2,N}(\wt\mu,\mu_{\beta}^{\eps})
    \geq 
    s \Big(
    \|\wh\mu-\wh\mu'\|_{\sTV}
    -\wh\mu(0)-\wh\mu'(0)\Big) \, .
    \end{equation}
    It will be convenient to discard the element $\{0\}$ by defining $\wh\mu_{-0},\wh\mu_{-0}'$ two probability measures on $\{1,2,\dots,M\}$ via
    \[
        \wh\mu_{-0}(i)=\frac{\wh\mu(i)}{1-\wh\mu(0)}\, ,
        \quad\quad
        \wh\mu_{-0}'(i)=\frac{\wh\mu'(i)}{1-\wh\mu'(0)}\, .
    \]
    ($\wh\mu_{-0}$ thus agrees with the definition \eqref{eq:wh-mu-def}.)
    It is easy to verify that 
    \begin{equation}\label{eq:restrict_to_clusters}
    \|\wh\mu_{-0}-\wh\mu_{-0}'\|_{\sTV} \le \frac{1}{1-\wh\mu(0)} \|\wh\mu-\wh\mu'\|_{\sTV}\, .
    \end{equation}
    Since $\wh\mu(0) \le e^{-cN}$ by assumption \ref{it:clusters-cover}, it suffices to lower-bound the average TV distance 
    \begin{equation}
    \label{eq:suppose-TV-small}
    \gamma := \E\big[\, \|\wh\mu_{-0}-\wh\mu_{-0}'\|_{\sTV}\, \big| \, \bG \big]\, .
    \end{equation}
    We assume $\gamma\leq 1/40$ below, without loss of generality.
    Let $S_{\same}$ be the subset of $1\leq i\leq M$ satisfying
    \begin{equation*}
    \lt|
    \frac{\wh\mu_{-0}'(i)}{\wh\mu_{-0}(i)}
    -
    1
    \rt|
    \leq 
    10\gamma \, .
    \end{equation*}
    Then by Markov's inequality,  
    \begin{equation}\label{eq:markov_S}
    \E \big[\wh\mu_{-0}(S_{\same}) \, \big| \, \bG\big] \geq 1-\frac{1}{10\gamma}\E \big[\|\wh\mu_{-0}-\wh\mu_{-0}'\|_{\sTV}\, \big| \, \bG\big] = 0.9\, .
    \end{equation}
    Next, denoting
    \[r_i = \frac{\wh\mu_{-0}'(i)}{\wh\mu_{-0}(i)}\,,~~~~~~R_{ij}' = \frac{\wh\mu_{-0}'(i)}{\wh\mu_{-0}'(j)}\,,~~~\mbox{and}~~~R_{ij} = \frac{\wh\mu_{-0}(i)}{\wh\mu_{-0}(j)} \, ,\]
    we have
    \begin{align*}
    \left|\frac{R_{ij}'}{R_{ij}} - 1\right| \le \left|\frac{r_i-1}{r_j}\right| + \left|\frac{r_j-1}{r_j}\right|\, .
    \end{align*}
    Under the event that \revv{$i,j\in S_{\same}$, i.e.}
    \[ |r_i-1| \le 10 \gamma\, , ~~~~\mbox{and}~~~~|r_j-1| \le 10 \gamma \, ,\]
    we have 
    \begin{equation}\label{eq:pairs_S}
    \left|\frac{R_{ij}'}{R_{ij}} - 1\right| \le \frac{20\gamma}{1-10\gamma} \, .
    \end{equation}
    Thus using~\eqref{eq:markov_S},   
    the set $S_{2,\same}\subseteq [M]^2$ of pairs $(i,j)$, $i \neq j$ which satisfy Eq.~\eqref{eq:pairs_S}
    must satisfy
    \begin{equation}
    \label{eq:pair-ratios-mostly-close}
    \E\big[\wh\mu_{-0}^{\otimes 2}(S_{2,\same})  \, \big| \, \bG\big]\geq 1 - 0.2 = 0.8 \, .
    \end{equation}
    Thus we have
    \begin{align}
    \nonumber
    0.8 
    &\le 
    \E\Big[ \wh\mu_{-0}^{\otimes 2}(S_{2,\same}) \, \big| \, \bG \Big] 
    \\
    \nonumber
    &= \sum_{1\le i\neq j\le M} \wh\mu_{-0}(i)\wh\mu_{-0}(j) \P\left(\Big|\frac{R_{ij}'}{R_{ij}} - 1\Big| \le \frac{20\gamma}{1-10\gamma} \, \Big|\, \bG\right)
    \\
    \label{eq:the-last-line}
    &\le \frac{\kappa}{c_0\eta\sqrt{N}} \, ,
    \end{align}
    where 
    \begin{equation}
    \label{eq:kappa_r}
    \kappa = 2 \max\Big(\log\Big(1+ \frac{20\gamma}{1-10\gamma}\Big), -\log\Big(1- \frac{20\gamma}{1-10\gamma}\Big)\Big)
    \revv{=
    -2\log\Big(1- \frac{20\gamma}{1-10\gamma}\Big)}
     ,
    \end{equation}
    and we used the anti-concentration result of Proposition~\ref{prop:cluster-pair-chaos} to obtain the last line \eqref{eq:the-last-line}. Since we assume $\gamma \le 1/40$, we have $\kappa \le 160\gamma$ and we obtain
    \[
        \E\big[\, \|\wh\mu_{-0}-\wh\mu_{-0}'\|_{\sTV}\, \big| \, \bG \big] 
        =
        \gamma
        \ge c_0 \eta \sqrt{N}/200.
    \]
    Plugging this into Eq.~\eqref{eq:restrict_to_clusters} and then Eq.~\eqref{eq:W2-to-TV} finishes the proof, assuming $\sep_{i,j,3r}$ for all $i,j$. 

To remove this extraneous assumption, instead of only considering the set $S_{2,\same}$ in Eq.~\eqref{eq:pair-ratios-mostly-close} we need to further restrict to the good event $\Omega^c$ of Proposition~\ref{prop:cluster-pair-chaos}. Recall that $\Omega^c$ is the set of distinct pairs  $(i,j)$ such that $L_{ij}$ is anti-concentrated. Rebooting the argument from Eq.~\eqref{eq:pair-ratios-mostly-close}, we instead have
    \begin{equation}
    \label{eq:pair-ratios-mostly-close_2}
    \E\big[\wh\mu_{-0}^{\otimes 2}(S_{2,\same} \cap \Omega^c)  \, \big| \, \bG\big]\geq 1 - 0.3 = 0.7 \, .
    \end{equation}
    Proceeding in the same fashion we then have
    \begin{align*}
    0.7 
    &\le 
    \E\Big[ \wh\mu_{-0}^{\otimes 2}(S_{2,\same}\cap \Omega^c) \, \big| \, \bG \Big] \\
    &= 
    \sum_{(i,j) \notin \Omega} \wh\mu_{-0}(i)\wh\mu_{-0}(j) \P\left(\Big|\frac{R_{ij}'}{R_{ij}} - 1\Big| \le \frac{20\gamma}{1-10\gamma} \, \Big|\, \bG\right) \\
    &\le 
    \frac{\kappa}{\eta\sqrt{N}} \, ,
    \end{align*}
where $\kappa$ is given in Eq.~\eqref{eq:kappa_r}. This completes the proof.
\end{proof}

\subsection{Lower bound for the onset of disorder chaos}
Now we prove the converse result, Proposition~\ref{prop:chaos-converse}, saying that if the perturbation is small, i.e, if $\eps_N = o(1/N)$, then the perturbed measure is close to the original one in total variation distance. Although we do not state it this way, the proof below also directly extends to arbitrary reference measures $\mu_0\in\cP(\cS_N)$. 
\revv{For convenience we take $\beta_N=\beta$ to be constant below.}
    
Define the auxiliary probability measure $\bar\mu_{\beta}$ by
    \[
    \de\bar\mu_{\beta}(\bsig)
    \propto
    e^{\beta (1-\eps_N)H_N(\bsig)}\de\mu_0(\bsig)\, .
    \]
    Using the notation of Eq.~\eqref{eq:correlated-disorder}, 
    namely $\bG^{\eps_N}=
(1-\eps_N)\bG+\eta_N\bW$,
$\eta_N =\sqrt{2\eps_N-\eps_N^2}$, and writing 
 $\wtH_N(\bsig):=N^{-(p-1)/2}\<\bW,\bsig^{\otimes p}\>$,
we also define the un-normalized positive measure 
    \begin{equation}
    \label{eq:unnormalized}
    \de\mu^{\dagger}_{\beta}(\bsig)
    =
    e^{\beta \eta_N \wtH_N(\bsig)}\de\bar\mu_{\beta}(\bsig)\, .
    \end{equation}

    First, we observe that 
    \begin{equation}
    \label{eq:dilation-TV}
    \bbE \|\mu_{\beta}-\bar\mu_{\beta}\|_{\sTV}\leq o(1).
    \end{equation}
    Indeed since $\sup_{\bsig\in\cS_N}|H_N(\bsig)|\leq C(p) N$ holds with exponentially good probability $1-e^{-cN}$
    (see Proposition \ref{prop:norm-bound}), the Radon--Nikodym derivative satisfies 
    \[
    \sup_{\bsig\in\cS_N}
    \lt|\frac{\de \mu_{\beta}}{\de\bar\mu_{\beta}}-1\rt|\leq 
    C'(p)\beta\eps_N N= o(1)\, ,
    \]
    with the same probability.
    
    Next, note that if $Z$ is a standard Gaussian and $\gamma_N = o(1)$ is a deterministic sequence, then
    \[
        \bbE\lt[
        |e^{\gamma_N Z}-1|
        \rt]
        = o(1)\, .
    \]
    Setting $\gamma_N=\eta_N\sqrt{N}$ immediately yields
    \[
    \bbE\lt[
        \big|e^{\beta \eta_N \wtH_N(\bsig)}-1\big|
        ~\Big|~
        \bG
        \rt]
        = o(1)\,.
    \]
    Recalling \eqref{eq:correlated-disorder} and \eqref{eq:unnormalized}, averaging the previous display over $\bsig\sim  \bar\mu_{\beta}$ implies that
    \[
    \bbE\|\mu^{\dagger}_{\beta} -\bar\mu_{\beta} \|_{\sTV} = o(1)\,.
    \]
    Finally the probability measure obtained by normalizing $\mu^{\dagger}_{\beta}$ is simply $\mu_{\beta}^{\eps_N}$, so the bound 
    \[
    \|\mu_{\beta}^{\eps_N} -\bar\mu_{\beta} \|_{\sTV}
    \leq 
    3
    \|\mu^{\dagger}_{\beta} -\bar\mu_{\beta} \|_{\sTV}
    \]
    holds almost surely. 
    Combining these estimates completes the proof.

\section{Proof of Theorem~\ref{thm:improved-clustering}}
\label{sec:Improved}

The key difference with Theorem~\ref{thm:main} is that we do not assume 
$1-\lbq\gg 1-\ubq$, which was the case there. This makes the construction of the
shattering decomposition more delicate and as we explain below, we no longer take the clusters to be spherical caps.

Given $\mu_{\beta}$, for sufficiently small constants $\alpha,\delta>0$ depending on
 $(p,\beta)$ and letting $r:=\sqrt{2(1-\ubq)}$, $R:=\sqrt{2(1-\lbq)}$ define:
\begin{align}
\label{eq:S-def}
  \wh S
  &:=
  \lt\{\bsig\in \cS_N~:~\mu_{\beta}(\Ball_r(\bsig))
  \geq 
  (1-e^{-\alpha N})\mu_{\beta}(\Ball_R(\bsig))\rt\},
  \\
  S
  &:=
  \Ball_{\delta}(\wh S)\, .
\end{align}
Assumption \ref{it:nonmonotone} on $\cF_{\beta}$
(namely $q\mapsto \cF_{\beta}(q)$ is strictly increasing on the interval
   $(\lbq,\ubq)$) implies
\begin{align}
\lim_{N\to\infty}\frac{1}{N}\E\log \frac{\mu_{\beta}(\Ball_R(\bsig)\setminus
\Ball_r(\bsig))}{\mu_{\beta}(\Ball_R(\bsig))} = -2c<0\, .
\end{align}
where expectation is taken with respect to $\bG$ and $\bsig$ (conditional on $\bG$,
$\bsig\sim\mu_{\beta}$).

Using a concentration
argument similar to in the proofs of Proposition \ref{prop:S-reg-covers} \revv{ (where one uses contiguity to argue within the planted model, and discretizes the relevant interval of overlaps via Proposition~\ref{prop:norm-bound})},
 it follows that 
\begin{equation}
\label{eq:S-big}
\bbE \mu_{\beta}(S)
\geq 
\bbE \mu_{\beta}(\wh S)
\geq 
1-e^{-cN}
\end{equation}
for suitably small $\alpha,\delta$ and smaller $c$.

In the proof of Theorem~\ref{thm:main}, we observed that when $r\ll R$, the set 
$S$ (analogous to $S_{\sreg}$ in the previous construction) deterministically satisfies an overlap gap property.
Namely, for any pair of points $\bsig_1,\bsig_2\in S$ we have $\|\bsig_1-\bsig_2\|_2/\sqrt{N}\not\in [3r,R/3]$, c.f.\ Lemma \ref{lem:OGP}. This immediately implied
a decomposition of $S$ into clusters. While this does not work directly when $r/R$ is close to $1$, we show that $S$ still obeys a useful geometric property. 
Namely, shortest distance paths in $S$ require exponentially large length to reach a significant distance from any given starting point. For example, $S$ cannot contain long line segments. In other words, the graph of nearest neighbors in $S$ is tortuous when embedded in the Euclidean sphere $\cS_N$. 

Recall the notation $\|\bx\|_N:=\|\bx\|/\sqrt{N}$ for $\bx\in\reals^N$.
\begin{definition}
  Given $X\subseteq\cS_N$, we define the \emph{$\delta$-adjacency graph $\cG_{X,\delta}$ 
  of $X$} as the undirected  graph with  vertex set $X$ and edge set 
  \[
  E(\cG_{X,\delta}):=\lt\{(\bx_1,\bx_2)\in \binom{X}{2}~:~\|\bx_1-\bx_2\|_N\leq \delta\rt\}\, .
  \]
  (Here,  $\binom{X}{2}$ is the set of unordered pairs in $X$.)
  For $x\in X\subseteq \cS_N$, let $\Ball_{\ell}(\bx;\cG_{X,\delta})\subseteq X$ denote the set 
  of points in $X$ whose graph distance from $\bx$ in 
  $\cG_{X,\delta}$ is at most $\ell$.

  Finally, we define a \emph{$(X,\delta)$-path} to be a path inside $\cG_{X,\delta}$, 
  i.e., a finite sequence $(\bx_1,\bx_2,\dots,\bx_k)\in X^k$ such that $\|\bx_{i+1}-\bx_i\|_N\leq \delta$ 
  for all $1\leq i\leq k-1$.
\end{definition}

\begin{proposition}
\label{prop:S-long-paths}
For suitably small constants $\delta,\alpha>0$, for all $\bs\in S$ we have 
  \[
  \Ball_{e^{\alpha N/2}}(\bs;\cG_{S,\delta})\subseteq \Ball_{3R}(\bs)\, .
  \]
\end{proposition}

\begin{proof}
  Suppose $\bs_1,\bs_2,\dots,\bs_k$ is an $(S,\delta)$-path such that $\|\bs_k-\bs_1\|_N\geq 3R$. 
  We will show that $k>e^{\alpha N/2}$.
  First, for each $i\le k$ let $\hbs_i\in \hS$ be such that $\|\hbs_i-\bs_i\|_N\leq \delta$. 
  Then $\hbs_1,\dots,\hbs_k$ is a $(\hS,3\delta)$-path such that $\|\hbs_k-\hbs_1\|_N\geq 5R/2$.

For any $\rho\ge 0$ and each $i\le k$, define $\psi_{i}(\bx;\rho):= \bfone(\|\bx-\hbs_i\|_N\le \rho)$ and
\begin{align}
\psi_*(\bx;\rho):= \sum_{i=1}^k \psi_{i}(\bx;\rho)\, .
\end{align}

Take $\delta\leq\frac{R-r}{10}$ small enough. It then follows using the discrete intermediate 
value theorem that for any $i\le k$ and any $\bx\in \Ball_r(\hbs_i)$ there exists $1\leq j\leq k$ such 
  that $r<\|\bx-\hbs_j\|<R$. This implies
  \begin{align}
  \psi_{i}(\, \cdot\, ;r)\le \sum_{j=1}^k\big[ \psi_{j}(\, \cdot\, ;R) -
   \psi_{j}(\, \cdot\, ;r)\big]\, .
   \end{align}
   Summing over $i$, we get
  \begin{align}  
  k\psi_*(\,\cdot\, ;R)-(k+1)\psi_*(\,\cdot\, ;r)\ge 0\, .
  \end{align}
  Integrating over all of $\cS_N$ with the Gibbs measure, we obtain
  \begin{align}  
  k\sum_{i=1}^{k}\mu_{\beta}\big(\Ball_R(\hbs_i)\setminus\Ball_r(\hbs_i)\big)
  \ge \sum_{i=1}^{k}\mu_{\beta}\big(\Ball_r(\hbs_i)\big)\, .
  \end{align}
  Recalling \eqref{eq:S-def}, 
  this implies $k>e^{\alpha N/2}$ as desired.
\end{proof}

Using Proposition~\ref{prop:S-long-paths}, we now show $\mu_{\beta}$ has a 
bottleneck around each $\bs\in S$.
\begin{proposition}
\label{prop:S-cuts}
Assume $\max_{\bsig\in\cS_N} |H_N(\bsig)|\le C_0 N$ for some constant $C_0$. 
  Then there exist constants
  $\alpha,\delta>0$ such that, for any $\bs\in S$ there exists $i\leq e^{\alpha N/2}$ such that 
  \[
  \mu_{\beta}(\Ball_{i}(\bs;\cG_{S,\delta}))\geq (1-e^{-\alpha N/3})\mu_{\beta}(\Ball_{i+1}(\bs;\cG_{S,\delta})).
  \]
\end{proposition}

\begin{proof}
    If the claim were false, then for $j=e^{\alpha N/2}$, telescoping the product gives
  \begin{align}
  \frac{\mu_{\beta}(\Ball_{j}(\bs;\cG_{S,\delta}))}{\mu_{\beta}(\Ball_{1}(\bs;\cG_{S,\delta}))}
  \geq 
  (1+e^{-\alpha N/3})^{j-1}
  \geq 
  \exp\big\{e^{cN}\big\}\, , ~~~ c = c(\alpha)>0\,.\label{eq:Telescope}
  \end{align}
  On the other hand, since $\bs\in S$, there exists $\hbs\in \hS\subseteq S$
  such that $\|\hbs-\bs\|_N\le \delta$. Letting $\bs_1:=(\hbs+\bs)/2$, we see that
   $\Ball_{\delta/2}(\bs_1)\subseteq\Ball_{\delta}(\bs)\cap \Ball_{\delta}(\hbs)$,
   and therefore $\Ball_{\delta/2}(\bs_1)\subseteq \Ball_{1}(\bs;\cG_{S,\delta})$.
   By the assumed bound on the Hamiltonian, we have 
   $\mu_{\beta}(\Ball_{\delta/2}(\bs_1))\geq e^{-CN}$ for a constant $C$. 
  Therefore 
   \begin{align*}
  \frac{\mu_{\beta}(\Ball_{j}(\bs;\cG_{S,\delta}))}{\mu_{\beta}(\Ball_{1}(\bs;\cG_{S,\delta}))}
  \le  \frac{1}{\mu_{\beta}(\Ball_{1}(\bs;\cG_{S,\delta}))}
  \le 
    e^{CN}\, .
  \end{align*}
 This contradicts Eq.~\eqref{eq:Telescope} for $N$ large enough, thus proving our claim.
\end{proof}

The next Proposition lets us transform $\cG_{S,\delta}$-connected clusters into path-connected subsets of $\cS_N$, while maintaining separation between distinct clusters. 
We remark that the argument does not obviously generalize to the Ising state space.
\begin{proposition}
\label{prop:line-segments}
  Suppose $X,Y\subseteq S\subseteq \cS_N$ are connected subsets under the graph structure 
  $\cG_{S,\delta}$  introduced above, 
  and there is no edge between $X$ and $Y$. Let $\overline X$ be the union of
   all shortest-path geodesics on $\cS_N$ between neighbors in $X$
   (i.e., between pairs $\bx_1,\bx_2\in X$ such that
   $\|\bx_1-\bx_2\|_N\le \delta$), and similarly for $\overline Y$. 
   
   Then $\overline X,\overline Y$ are path-connected subsets of $\cS_N$ with $\overline X\subseteq \Ball_{2\delta/3}(X)$, $\overline Y\subseteq \Ball_{2\delta/3}(Y)$. 
   Moreover, $d(\overline X,\overline Y)\geq \delta\sqrt{N}/C$ for some universal constant 
   $C$ (independent of $\delta$).
\end{proposition}

\begin{proof}
  The path-connectivity and containment of $\overline X,\overline Y$ are clear.  
  Hence, we are left with the task of proving the separation 
  $d(\overline X,\overline Y)\geq \delta\sqrt{N}/C$.
  
   For the main part of the proof, 
  let $\wh X,\wh Y\subseteq \reals^N$ be defined in analogously to  $\overline X,\overline Y$
   but using line segments in
   $\bbR^N$ instead of geodesics in $\cS_N$. Formally,
   \begin{align*}
   \wh X:= \bigcup\limits_{\substack{\bx_1,\bx_2\in X,\\\|\bx_1-\bx_2\|_N\le \delta}}\,  L(\bx_1,\bx_2)\,
   ;
    \;\;\;\;\;\;\;\;\;\;
    L(\bx_1,\bx_2):= \big\{(1-t)\bx_1+t\bx_2:\; t\in [0,1]\big\}\,.
   \end{align*}
  The Hausdorff distances from this approximation satisfy $d_H(\wh X,\overline X), 
  d_H(\wh Y,\overline Y)\leq O(\delta^2)$, so it will suffice to show below that 
  $d(\wh X,\wh Y)\geq \delta\sqrt{N}/C$. 
  
  By contradiction, assume $\hbx\in \wh X$, $\hby'\in\wh Y$ are such that
  $\|\hbx-\hby'\|_N<\delta/C$. 
  By assumption, $\hbx=(1-t_x)\bx_0+t_x \bx_1\in \wh X$ is on 
  the line-segment connecting $\bx_0,\bx_1\in X$ for $t_x\in [0,1/2]$ and similarly 
  $\hby'$ is on the line segment connecting $\by_0'$, $\by_1'\in Y$. 
  Let $\hby = \hby'-(\hby'-\hbx)$, on the line segment connecting 
  $\by_0:=\by_0'-(\hby'-\hbx)$, $\by_1:=\by_1'-(\hby'-\hbx)$.
  Applying Stewart's theorem to the triangle $\triangle \bx_0 \by_0 \by_1$ with 
  cevian $\overline{\bx_0 \hbx}$ we get:
  \begin{equation}
  \label{eq:stewart}
  \|\bx_0-\by_0\|_N^2 \cdot \|\by_1-\hby\|_N
  +
  \|\bx_0-\by_1\|_N^2 \cdot \|\by_0-\hby\|_N
  =
  \|\by_0-\by_1\|_N\cdot \big(
  \|\bx_0-\hbx\|_N^2 + \|\by_0-\hby\|_N\cdot \|\by_1-\hby\|_N
  \big)\, .
  \end{equation}
  By definition, we have
  \begin{align*}
  &\|\bx_0-\by_0\|_N,\|\bx_0-\by_1\|_N
   \geq 
  (1-C^{-1})\delta\, ,
  \\
  &\|\by_1-\hby\|_N+\|\by_0-\hby\|_N
  \le 
  \delta
  \;\;\;
  \implies\;\; 
  \|\by_0-\hby\|_N\cdot \|\by_1-\hby\|_N
  \leq 
  \frac{\delta^2}{4}\, ,
  \\
  &\|\by_0-\by_1\|_N\leq \delta\, ,
  \\
  &\|\bx_0-\hbx\|_N
  =
  t_x \|\bx_0-\bx_1\|_N
  \leq 
  \frac{\delta}{2}\, .
  \end{align*}
  Therefore the left-hand side of \eqref{eq:stewart} is 
  \begin{align*}
  \|\bx_0-\by_0\|_N^2 \cdot \|\by_1-\hby\|_N
  +
  \|\bx_0-\by_1\|_N^2 \cdot \|\by_0-\hby\|_N\ge (1-C^{-1})^2\delta^3\, .
  \end{align*}
   Meanwhile the right-hand side is
   \begin{align*}
    \|\by_0-\by_1\|_N\cdot \big(
  \|\bx_0-\hbx\|_N^2 + \|\by_0-\hby\|_N\cdot \|\by_1-\hby\|_N\big)
\le \delta\Big(\frac{\delta^2}{4}+\frac{\delta^2}{4}\Big) = \frac{\delta^3}{2}\, .
  \end{align*}
  The two displays yield a contradiction if we take, e.g., $C=4$, thus proving our claim.
\end{proof}

\begin{proof}[Proof of Theorem~\ref{thm:improved-clustering}]
  We will first construct an initial clustering which satisfies all 
  properties required by Definition \ref{def:Shattering} except for 
  point \ref{it:clusters-isolated} on bottlenecks. 
  We will then remove violating clusters.

  We begin by combining Propositions~\ref{prop:S-long-paths} and \ref{prop:S-cuts}.
  On the high probability event \\ $\max_{\bsig\in\cS_N}|H_N(\bsig)|\le C_0N$, for $C_0$ large enough,
  the following holds.
   For any $\bs\in S$ there exists $i\leq e^{\alpha N/2}$ such that $\wh\cC_{\bs}:= 
  \Ball_{i}(\bs;\cG_{S,\delta})$ satisfies 
  \begin{align}
  \label{eq:clusters-diam-6R}
  \diam(\wh\cC_{\bs})&\leq 6R\, ,
  \\
  \label{eq:clusters-efficiently-extracted}
  \mu_{\beta}(\wh\cC_{\bs})&\geq (1-e^{-\alpha N/3})\mu_{\beta}(S\cap \Ball_{\delta}(\wh\cC_{\bs}))\, .
  \end{align}
  Moreover, it is easy to see that the same properties hold if $S$ is replaced by any 
  subset $S^{\circ}\subseteq S$ (in particular,
  this is the case for both Propositions~\ref{prop:S-long-paths} and \ref{prop:S-cuts}.)

  We form our clusters iteratively. 
  We start by setting $S_1=S$, and proceed as follows for $\ell\in\{1,2,\dots\}$:
  \begin{enumerate}
  \item Choose $\bs\in S_{\ell}$.
\item   Let the $\ell$-th cluster be $\cC_{\ell}=\overline{\cC}_{\bs}$,
 where  $\overline{\cC}_{\bs}\supseteq\wh\cC_{\bs}$, is  the path-connected superset 
  defined as in Proposition~\ref{prop:line-segments} using $S_{\ell}$ for $S$ therein. (In particular $\overline{\cC}_{\bs}\subseteq \Ball_{2\delta/3}(S_{\ell})$.)
\item   Set 
  $S_{\ell+1}= S_{\ell}\setminus \Ball_{\delta}(\wh{\cC}_{\bs})$. 
  \end{enumerate}
  This is repeated until $S_{\ell+1}$ becomes 
  empty. The number of iterations $M$ is at most the $\delta\sqrt{N}/3$-covering number of 
  $S\subseteq\cS_N$, and in particular is finite. This procedure can obviously be modified as to be  
  origin-symmetric if $p$ is even and $R\leq 0.01$.
  Note that we may have $\overline{\cC}_s \subsetneq S$, but this poses no issues. 

  We now check that all conditions of Definition \ref{def:Shattering}  besides
  point  \ref{it:clusters-isolated} are satisfied. 
  
  Consider first Condition~\ref{it:clusters-cover}.
  With $\cI$ the set of points $\bs\in S$ selected during 
  the procedure to construct clusters defined above (and used as subscripts),
  we have
  \begin{align}
\nonumber
  \bbE \mu_{\beta}
  \lt(
  \bigcup_{\ell=1}^M
  \cC_{\ell}
  \rt)
  &=
  \bbE \mu_{\beta}
  \lt(
  \bigcup_{\bs\in\cI}
  \overline{\cC_{\bs}}
  \rt)\\
  \nonumber
&  \ge
  \bbE \mu_{\beta}
  \lt(
  \bigcup_{\bs\in\cI}
    \wh\cC_{\bs}
  \rt)\\
  \nonumber
  &  \stackrel{(a)}{\geq} 
  (1-e^{-\alpha N/3})
  \mu_{\beta}(S)\\
  \label{eq:most-mass-inside}
  &\stackrel{(b)}{\geq}  1-e^{-cN/2}.
  \end{align}
Here $(a)$ follows from Eq.~\eqref{eq:clusters-efficiently-extracted} and $(b)$
from Eq.~\eqref{eq:S-big}.

Condition~\ref{it:small-diam} is clear from \eqref{eq:clusters-diam-6R}. 
In order to prove Condition \ref{it:small-prob}, note that
  \begin{align*}
  \mu_{\beta}^{\otimes 2}\Big(\Big\{\bsig,\bsig' \,:\,
  \la \bsig,\bsig'\ra/N \geq 0.1
  \Big\}\Big)
  \geq 
  \sum_{\bs\in\cI} \mu_{\beta}(\overline{\cC_\bs})^2\, .
  \end{align*}
  However using Corollary~\ref{cor:FP-maximized}, it follows that for some constant $c_0>0$,
  \begin{align*}
  \mu_{\beta}^{\otimes 2}\Big(\Big\{\bsig,\bsig' \,:\,
  \la \bsig,\bsig'\ra/N \geq 0.1
  \Big\}\Big)\le e^{-c_0N}\, ,
  \end{align*}
    holds with exponentially high probability. The previous two displays imply the claim.
  
  Condition~\ref{it:clusters-separate} is also clear by construction: for any $k<\ell$, the set $S_{\ell}$ is disjoint from $\Ball_{\delta}(\cC_{k})$, while $\cC_{\ell}\subseteq \Ball_{2\delta/3}(S_{\ell})$ by Proposition~\ref{prop:line-segments}.

  We now refine our clustering to ensure Condition~\ref{it:clusters-isolated},
   by deleting all clusters which violate
   it\footnote{Condition~\ref{it:clusters-isolated} might fail to hold even for the first
    cluster: the right-hand side of \eqref{eq:clusters-efficiently-extracted} has 
    $S\cap \Ball_{\delta}(\wh\cC_s)$ rather than just $\Ball_{\delta}(\wh\cC_s)$. Further issues 
    arise for later clusters because the validity of Condition~\ref{it:clusters-isolated} depends on
     the mass deleted during previous iterations.}. More specifically, we delete any 
     $\overline{\cC_{\bs}}$ formed above which satisfies
  \begin{equation}
  \label{eq:non-bottleneck}
    \mu_{\beta}(\overline{\cC_{\bs}})\leq (1-e^{-cN/10})\mu_{\beta}(\Ball_{\delta/10}
    (\overline{\cC_{\bs}})).
  \end{equation}
  Note that the neighborhoods $\Ball_{\delta/10}(\overline{\cC_{\bs}})$ are disjoint 
  for different clusters. Let $\cD\subseteq\cI$ denote the set of deleted clusters, and set
  \[
  \eps
  :=
  \mu_{\beta}
  \lt(
  \bigcup_{s\in\cD}
  \overline{\cC_{\bs}}
  \rt)\, .
  \]
  Then \eqref{eq:non-bottleneck} and disjointness imply
  \[
  \mu_{\beta}
  \lt(
  \bigcup_{s\in\cD}
  \Ball_{\delta/10}(\overline{\cC_s})\backslash\overline{\cC_s}
  \rt)
  \geq 
  e^{-cN/10} \eps\, .
  \]
  However, \eqref{eq:most-mass-inside} implies that with exponentially good probability,
  \[
   \mu_{\beta}
  \lt(
  \bigcup_{s\in\cD}
  \Ball_{\delta/10}(\overline{\cC_s})\backslash\overline{\cC_s}
  \rt)
  \leq e^{-cN/2}.
  \]
  Hence on this event, we find that $\eps\leq e^{-cN/10}$.
  
   By definition of $\eps$, this means Condition~\ref{it:clusters-cover} continues to hold 
   after the deletions. Of course, Condition~\ref{it:clusters-isolated} trivially
    holds after deletions. Finally, Assertions~\ref{it:small-diam}, \ref{it:small-prob}, 
    and \ref{it:clusters-separate} are automatically inherited since the new clusters 
    constitute a subset of the preliminary clusters.
\end{proof}

\section*{Acknowledgments}

AM was supported by the NSF through award DMS-2031883, the Simons Foundation through
Award 814639 for the Collaboration on the Theoretical Foundations of Deep Learning, the NSF
grant CCF-2006489 and the ONR grant N00014-18-1-2729. Part of this work was carried out while
Andrea Montanari was on partial leave from Stanford and a Chief Scientist at Ndata Inc dba
Project N. The present research is unrelated to AM’s activity while on leave.
We thank Brice Huang, Eren K{\i}z{\i}lda{\u{g}}, Kevin Luo, Xiaodong Yang, and the anonymous referees for helpful comments.

\appendix

\section{The static and dynamical temperatures}
\label{app:StaticTemperature}

Here we study $\beta_d(p)$, $\beta_c(p)$, and in particular derive the
asymptotic formulas 
\begin{equation}
\beta_d(p) =\sqrt{e}+O(1/p),\;\;\;\;\; \beta_c(p) = \sqrt{\log p}\cdot
(1+o(1))\, ,\label{eq:AsymBeta}
\end{equation}
(as $p\to\infty$) that are quoted in the main text. 

\paragraph{The static temperature.} To characterize $\beta_c(p)$, we use the Crisanti-Sommers formula of
Eq.~\eqref{eq:CrisantiSommers} (see
\cite{jagannath2018bounds}). 

For the case $\xi'(0)=0$, the high-temperature phase $\beta\leq \beta_c(p)$ is defined by $\zeta=\delta_0$ being
the minimizer of this functional, which we denote below $\Par(\zeta):=\Par_{\beta}(\zeta;\xi)$. Such a minimizer is always unique as $\Par(\zeta)$ is strictly convex.
Furthermore, the derivative of $\Par(\zeta)$ with respect to adding an atom at $\delta_q$ is
\begin{align*}
\lim_{\eps\downarrow 0}
\frac{1}{\eps}\left[\Par((1-\eps)\delta_0+\eps\delta_q)-\Par(\delta_0)\right] &=
    -\frac{\beta^2}{2}\int_0^q (q-t)\xi''(t)~\de t
    +\frac{1}{2}
    \int_0^q 
    \frac{q-t}{(1-t)^2}~\de t\\
    &=-\frac{1}{2}\big[\beta^2\xi(q)
    +q
    +\log(1-q)\big].
\end{align*}
Hence we have (in agreement with \cite[Proposition 2.3]{talagrand2006free}):
\[
    \beta\leq \beta_c(p)
    \iff
    \inf_{q\in (0,1)}
    \Big(
    -\beta^2
    \xi(q)
    -q
    +\log\lt(\frac{1}{1-q}\rt)
    \Big)
    \geq 0.
\]
Equivalently, for the case $\xi(t) = t^p$ of interest
\begin{equation}
    \beta_c^2(p)
    =
    \inf_{q\in [0,1]}
    \lt\{\frac{1}{q^p}\log\lt(\frac{1}{1-q}\rt)-\frac{1}{q^{p-1}}\rt\}.\label{eq:BetacInf}
\end{equation}
To get an upper bound, we choose $q=1-(x/p)$ which yields, as $p\to\infty$:
\begin{align*}
 \beta_c^2(p)
    \le e^x\log p +O(1)
\end{align*}
Taking $x\to 0$ after $p\to \infty$ we thus get 
$ \beta_c(p)\le \sqrt{\log p}\cdot (1+o_p(1))$ for $p$ large.

To get a matching lower bound, we split the infimum in Eq.~\eqref{eq:BetacInf}
in two regions:
\begin{enumerate}
\item For $0<q<1-c(\log p)/p$, we use $-\log(1-q)-q\ge q^2/2$ to get
\begin{align*}
\frac{1}{q^p}\log\lt(\frac{1}{1-q}\rt)-\frac{1}{q^{p-1}}&\ge \frac{1}{2q^{p-2}}\\
&\ge \frac{1}{2}\Big(1-c\frac{\log p}{p}\Big)^{-p+2}\\
& = \frac{1}{2}\cdot 
\exp\Big(c\log p+O\big((\log p)^2/p\big)\Big)\ge p^{c/2}\, ,
\end{align*}
where the last inequality holds for all $p$ large enough.
\item For $1-c(\log p)/p<q<1$, we set $q= 1-x/p$ with $x\leq c\log p$ 
and obtain:
\begin{align*}
\frac{1}{q^p}\log\lt(\frac{1}{1-q}\rt)-\frac{1}{q^{p-1}}&
= \frac{1}{q^p}\lt(\lt(\log \frac{p}{x}\rt)-q\rt)\\
&\ge \log p - \log(c \log p)-1.
\end{align*}
%
\end{enumerate}
Taken together, these bounds imply $\beta_x(p)\ge \sqrt{\log p}\cdot (1-o(1))$
and therefore conclude the proof of the asymptotic formula for
$\beta_c(p)$ in Eq.~\eqref{eq:AsymBeta}.

\revv{ \paragraph{The dynamical temperature.} As
discussed in the introduction, the dynamical/shattering phase transition
temperature was first identified\footnote{This paper
contains a typo, corrected e.g. in \cite{castellani2005spin} or \cite[Appendix A]{folena2020rethinking}.
} in \cite{crisanti1993spherical}
via the asymptotics of the equilibrium correlation
function $C(t) = \lim_{N \to \infty} \langle\bsig_0 , \bsig_t\rangle /N$ (assuming the limit exists) where $\bsig_0 \sim \mu_{\beta}$.}

\revv{For general mixed $p$-spin model with mixture $\xi$, such that $\xi'(0)=0$,
the plateau $q = \lim_{t \to \infty} C(t)$ is expected to satisfy~\cite{cugliandolo2004course} the equation      
\begin{equation}
\beta^2\xi'(q)(1-q) = q \, .
\end{equation}
The inverse temperature for the 
dynamical phase transition  $\beta_d(\xi)$ is the smallest $\beta$ at which a non-zero solution 
appears:
\begin{align}
\beta_d(\xi)= \inf_{q\in (0,1)} \frac{q}{\xi'(q)(1-q)}\,.
\end{align}
In the case of a pure $p$-spin $\xi(t) = t^p$, a simple calculation yields the formula 
\begin{align}
\beta_d(p)=\sqrt{\frac{(p-1)^{p-1}}{p(p-2)^{p-2}}}\,.
\end{align}
}      
The claim of Eq.~\eqref{eq:AsymBeta} follows by calculus.

\section{Overlap disorder chaos in the replica-symmetric phase}
\label{sec:trivial-overlaps}

Here we consider general mixed $p$-spin models without external field, whose covariances are parametrized by $\xi(t)=\sum_{p\geq 2} \gamma_p^2 t^p$ for an exponentially decaying sequence $\gamma_2,\gamma_3,\dots \geq 0$. It will be useful to consider the perturbed mixture functions $\xi^{(p,\delta)}(t)=\xi(t)+\delta t^p$, where $\delta$ is small but may be negative.

\begin{proposition}
\label{prop:RS-stability}
    If $\beta<\beta_c(\xi)$, then for any $p$ with $\gamma_p>0$ the free energy of $\xi^{(p,\delta)}$ is replica symmetric for $|\delta|< \delta_0(\xi,\beta,p)$ small enough.
\end{proposition}

\begin{proof}
    In fact we may take $\delta_0 = \gamma_p^2\cdot \min(1,\beta_c-\beta)$. For instance, \cite[Corollary 2.1]{sellke2023free} shows that if $\wt\xi\leq \wh\xi$ holds coefficient-wise, then $F_{\beta}(\wh\xi)=F_{\beta}^{\Ann}(\wh\xi)$ implies the same for $\wt\xi$.
\end{proof}

\begin{proposition}
\label{prop:RS-trivial-overlaps}
    Let $\xi$ be any mixture function.
    Then for $\beta<\beta_c(\xi)$, letting
    $\bsig,\bsig'\stackrel{i.i.d.}{\sim}\mu_{\beta}$,
    \[
    \lim_{N\to\infty}
    \bbE\big[|\la\bsig,\bsig'\ra/N|^2\big]=0.
    \]
    Moreover the same holds in the Ising case.
\end{proposition}


\begin{proof}
    By definition, $F_{\beta}(\xi)=\beta^2 \xi(1)/2$ for $\beta\leq\beta_c$.
    Fix $p$ with $\gamma_p>0$.
    Then for $|\delta|$ small enough, as $N\to\infty$, $F_{N,\beta}(\xi^{(p,\delta)})$ tends to 
    \[
    F_{\beta}(\xi^{(p,\delta)})=\frac{\beta^2 (\xi(1)+\delta)}{2}
    \]
    locally uniformly in probability; the right-hand value follows from Proposition~\ref{prop:RS-stability}.
    Therefore 
    \[
    \plim_{N\to\infty}\frac{\de}{\de\delta}F_{N,\beta}(\xi^{(p,\delta)})\big|_{\delta=0}
    =
    \beta.
    \]
    On the other hand \cite[Theorem 1.2]{talagrand2006parisimeasures} (see also \cite[Theorem 3.7]{panchenko2013sherrington}) gives
    \begin{align}
    \nonumber
    \frac{\de}{\de\delta}F_{N,\beta}(\xi^{(p,\delta)})\big|_{\delta=0}
    &=
    \beta\big(1-\bbE\big[\big(\la\bsig,\bsig'\ra/N\big)^p\big]\big),
    \\
    \label{eq:xi-zero-on-average}
    &\implies 
    \bbE\big[\big(\la\bsig,\bsig'\ra/N\big)^p\big]
    =
    o_N(1).
    \end{align}
    In the case that $\gamma_p>0$ for some even $p$, this completes the proof. 

    In the case that $\xi$ is an odd function, note that for any $\eps>0$ there exists $c=c(\xi,\eps)>0$ such that $\xi(t)\leq -c$ on $t\in [-1,-\eps]$. Let 
    \[
    \cS_N(2,[-1,-\eps])
    =
    \lt\{
    (\bsig,\bsig')\in\cS_N^2~:~
    \frac{\la\bsig,\bsig'\ra}{N}\in [-1,-\eps].
    \rt\}
    \]
    Then for $(\bsig,\bsig')\in\cS_N(2,[-1,-\eps])$, the sum $H_N(\bsig)+H_N(\bsig')$ is a centered Gaussian with 
    \[
    \frac{1}{N}\bbE\lt[\big(H_N(\bsig)+H_N(\bsig')\big)^2\rt]
    =
    2\xi(1)-2\xi(\la\bsig,\bsig'\ra/N)<2\xi(1)-2c.
    \]
    By Jensen's inequality at finite $N$, we have the annealed upper bound
    \begin{align*}
    \lim_{N\to\infty}
    \frac{1}{N}
    \bbE\log
    \int
    e^{\beta H_N(\bsig)+\beta H_N(\bsig')}
    1_{\cS_N(2,[-1,-\eps])}
    ~\de \mu_0^{\otimes 2}
    &\leq 
    \lim_{N\to\infty}
    \frac{1}{N}
    \log
    \int
    \bbE
    e^{\beta H_N(\bsig)+\beta H_N(\bsig')}
    1_{\cS_N(2,[-1,-\eps])}
    ~\de \mu_0^{\otimes 2}
    \\
    &\leq 
    \lim_{N\to\infty}
    \frac{1}{N}
    \log
    \int_{\cS_N^2}
    e^{N\beta^2\xi(1)-Nc\beta^2}
    ~\de \mu_0^{\otimes 2}
    \\
    &\leq 
    \beta^2\xi(1)-c\beta^2.
    \end{align*}
    Since $\beta<\beta_c$, standard free energy concentration implies with probability $1-o_N(1)$ over $\bG^{(p)}$, 
    \[
    \mu_{\beta}^{\otimes 2}(\cS_N(2,[-1,-\eps])\leq o_N(1).
    \]
    Combining with \eqref{eq:xi-zero-on-average} implies that the overlap also concentrates around $0$ for odd $\xi$.
\end{proof}

\begin{proposition}
\label{prop:OverlapChaos}
    Let $\xi$ be any mixture function.
    Let $\mu_{\beta}$, $\mu^{\eps_N}_{\beta}$ be the corresponding
    $p$-spin model Gibbs measures with $\eps_N$-correlated disorder
    (either spherical or Ising).
    Then for any $\beta<\beta_c(\xi)$ and $\eps_N\geq 0$, with $\bsig,\bsig'\sim\mu_{\beta}\otimes \mu^{\eps_N}_{\beta}$,
    we have
    \[
    \lim_{N\to\infty}
    \bbE[|\la\bsig,\bsig'\ra/N|^2]=0.
    \]
    Moreover, this limit holds uniformly in $\eps_N\in [0,1]$. 
\end{proposition}

\begin{proof}
    This follows  from Proposition~\ref{prop:RS-trivial-overlaps}    
    by writing, for $(\bsig,\bsig')\sim\mu_{\beta}\times \mu_{\beta}'$:
    \[
    \bbE\big[\big(\la\bsig,\bsig'\ra/N\big)^2\big]
    =
    \bbE
    \lt\la
    \bbE_{\bx\sim\mu_{\beta}} [\bx^{\otimes 2}]/N
    \,,\,
    \bbE_{\bx'\sim\mu_{\beta'}} [\bx'^{\otimes 2}]/N
    \rt\ra.
    \]
    Here we use the usual matrix inner product $\la A,B\ra=\Tr(A^{\top}B)$. Hence the extension of Proposition~\ref{prop:RS-trivial-overlaps} to Gibbs measures for different disorder is immediate by Cauchy--Schwarz.
\end{proof}

\footnotesize
\bibliographystyle{alpha}
\bibliography{bib}

\newcommand{\etalchar}[1]{$^{#1}$}
\begin{thebibliography}{KMRT{\etalchar{+}}07}

\bibitem[ACO08]{achlioptas2008algorithmic}
Dimitris Achlioptas and Amin Coja-Oghlan.
\newblock Algorithmic barriers from phase transitions.
\newblock In {\em 2008 49th Annual IEEE Symposium on Foundations of Computer
  Science}, pages 793--802. IEEE, 2008.

\bibitem[AMS22]{alaoui2022sampling}
Ahmed~El Alaoui, Andrea Montanari, and Mark Sellke.
\newblock Sampling from the sherrington-kirkpatrick gibbs measure via
  algorithmic stochastic localization.
\newblock In {\em 2022 IEEE 63rd Annual Symposium on Foundations of Computer
  Science (FOCS)}, pages 323--334. IEEE, 2022.

\bibitem[AMS23]{alaoui2023sampling}
Ahmed~El Alaoui, Andrea Montanari, and Mark Sellke.
\newblock Sampling from mean-field gibbs measures via diffusion processes.
\newblock {\em arXiv:2310.08912}, 2023.

\bibitem[ART06]{achlioptas2006solution}
Dimitris Achlioptas and Federico Ricci-Tersenghi.
\newblock On the solution-space geometry of random constraint satisfaction
  problems.
\newblock In {\em Proceedings of the thirty-eighth annual ACM symposium on
  Theory of computing}, pages 130--139, 2006.

\bibitem[Ban38]{banach1938homogene}
Stefan Banach.
\newblock {{\"U}ber homogene Polynome in $L^2$}.
\newblock {\em Studia Mathematica}, 7(1):36--44, 1938.

\bibitem[BDG06]{ben2006cugliandolo}
G{\'e}rard Ben\space{}Arous, Amir Dembo, and Alice Guionnet.
\newblock {Cugliandolo-Kurchan equations for dynamics of spin-glasses}.
\newblock {\em Probability theory and related fields}, 136(4):619--660, 2006.

\bibitem[BJ18]{ben2018spectral}
G{\'e}rard Ben\space{}Arous and Aukosh Jagannath.
\newblock Spectral gap estimates in mean field spin glasses.
\newblock {\em Communications in Mathematical Physics}, 361(1):1--52, 2018.

\bibitem[BJ24]{arous2021shattering}
G{\'e}rard Ben\space{}Arous and Aukosh Jagannath.
\newblock Shattering versus metastability in spin glasses.
\newblock {\em Communications on Pure and Applied Mathematics}, 77(1):139--176,
  2024.

\bibitem[BM87]{bray1987chaotic}
Alan~J Bray and Michael~A Moore.
\newblock Chaotic nature of the spin-glass phase.
\newblock {\em Physical review letters}, 58(1):57, 1987.

\bibitem[BMW00]{biroli2000variational}
Giulio Biroli, Remi Monasson, and Martin Weigt.
\newblock A variational description of the ground state structure in random
  satisfiability problems.
\newblock {\em The European Physical Journal B-Condensed Matter and Complex
  Systems}, 14(3):551--568, 2000.

\bibitem[CC05]{castellani2005spin}
Tommaso Castellani and Andrea Cavagna.
\newblock Spin-glass theory for pedestrians.
\newblock {\em Journal of Statistical Mechanics: Theory and Experiment},
  2005(05):P05012, 2005.

\bibitem[Cel24]{celentano2022sudakov}
Michael Celentano.
\newblock {Sudakov--Fernique post-AMP, and a new proof of the local convexity
  of the TAP free energy}.
\newblock {\em The Annals of Probability}, 52(3):923--954, 2024.

\bibitem[Cha09]{chatterjee2009disorder}
Sourav Chatterjee.
\newblock Disorder chaos and multiple valleys in spin glasses.
\newblock {\em arXiv preprint arXiv:0907.3381}, 2009.

\bibitem[Che13]{chen2013aizenman}
Wei-Kuo Chen.
\newblock {The Aizenman-Sims-Starr scheme and Parisi formula for mixed $p$-spin
  spherical models}.
\newblock {\em Electronic Journal of Probability}, 18, 2013.

\bibitem[Che19]{chen2019phase}
Wei-Kuo Chen.
\newblock {Phase transition in the spiked random tensor with Rademacher prior}.
\newblock {\em The Annals of Statistics}, 47(5):2734--2756, 2019.

\bibitem[CHL21]{chen2021phase}
Wei-Kuo Chen, Madeline Handschy, and Gilad Lerman.
\newblock Phase transition in random tensors with multiple independent spikes.
\newblock {\em The Annals of Applied Probability}, 31(4):1868--1913, 2021.

\bibitem[CHS93]{crisanti1993spherical}
Andrea Crisanti, Heinz Horner, and H~J Sommers.
\newblock {The spherical $p$-spin interaction spin-glass model: the dynamics}.
\newblock {\em Zeitschrift f{\"u}r Physik B Condensed Matter}, 92:257--271,
  1993.

\bibitem[CK93]{cugliandolo1993analytical}
Leticia~F Cugliandolo and Jorge Kurchan.
\newblock Analytical solution of the off-equilibrium dynamics of a long-range
  spin-glass model.
\newblock {\em Physical Review Letters}, 71(1):173, 1993.

\bibitem[CK94]{cugliandolo1994out}
Leticia~F. Cugliandolo and Jorge Kurchan.
\newblock {On the out-of-equilibrium relaxation of the Sherrington-Kirkpatrick
  model}.
\newblock {\em Journal of Physics A: Mathematical and General}, 27(17):5749,
  1994.

\bibitem[CS92]{crisanti1992spherical}
Andrea Crisanti and H-J Sommers.
\newblock The spherical $p$-spin interaction spin glass model: the statics.
\newblock {\em Zeitschrift f{\"u}r Physik B Condensed Matter}, 87(3):341--354,
  1992.

\bibitem[CS95]{crisanti1995thouless}
Andrea Crisanti and H-J Sommers.
\newblock {Thouless-Anderson-Palmer approach to the spherical p-spin spin glass
  model}.
\newblock {\em Journal de Physique I}, 5(7):805--813, 1995.

\bibitem[Cug04]{cugliandolo2004course}
Leticia~F Cugliandolo.
\newblock Course 7: Dynamics of glassy systems.
\newblock In {\em Slow Relaxations and nonequilibrium dynamics in condensed
  matter: Les Houches Session LXXVII, 1-26 July, 2002}, pages 367--521.
  Springer, 2004.

\bibitem[FFRT20]{folena2020rethinking}
Giampaolo Folena, Silvio Franz, and Federico Ricci-Tersenghi.
\newblock Rethinking mean-field glassy dynamics and its relation with the
  energy landscape: The surprising case of the spherical mixed p-spin model.
\newblock {\em Physical Review X}, 10(3):031045, 2020.

\bibitem[FP95]{franz1995recipes}
Silvio Franz and Giorgio Parisi.
\newblock Recipes for metastable states in spin glasses.
\newblock {\em Journal de Physique I}, 5(11):1401--1415, 1995.

\bibitem[GJ21]{gamarnik2019overlap}
David Gamarnik and Aukosh Jagannath.
\newblock The overlap gap property and approximate message passing algorithms
  for $ p $-spin models.
\newblock {\em The Annals of Probability}, 49(1):180--205, 2021.

\bibitem[GJK23]{gamarnik2023shattering}
David Gamarnik, Aukosh Jagannath, and Eren~C K{\i}z{\i}lda{\u{g}}.
\newblock {Shattering in the Ising Pure $p$-Spin Model}.
\newblock {\em arXiv preprint arXiv:2307.07461}, 2023.

\bibitem[GJW20]{gamarnik2020optimization}
David Gamarnik, Aukosh Jagannath, and Alexander~S. Wein.
\newblock Low-degree hardness of random optimization problems.
\newblock In {\em Proceedings of 61st FOCS}, pages 131--140. IEEE, 2020.

\bibitem[HMP24]{huang2024sampling}
Brice Huang, Andrea Montanari, and Huy~Tuan Pham.
\newblock {Sampling from Spherical Spin Glasses in Total Variation via
  Algorithmic Stochastic Localization}.
\newblock {\em arXiv preprint arXiv:2404.15651}, 2024.

\bibitem[HS23]{huang2023algorithmic}
Brice Huang and Mark Sellke.
\newblock Algorithmic threshold for multi-species spherical spin glasses.
\newblock {\em arXiv preprint arXiv:2303.12172}, 2023.

\bibitem[HS24]{huang2021tight}
Brice Huang and Mark Sellke.
\newblock {Tight Lipschitz hardness for optimizing mean field spin glasses}.
\newblock {\em Communications on Pure and Applied Mathematics}, 2024.

\bibitem[JLM20]{jagannath2020statistical}
Aukosh Jagannath, Patrick Lopatto, and L{\'e}o Miolane.
\newblock {Statistical thresholds for tensor PCA}.
\newblock {\em Annals of Applied Probability}, 30(4):1910--1933, 2020.

\bibitem[JT18]{jagannath2018bounds}
Aukosh Jagannath and Ian Tobasco.
\newblock Bounds on the complexity of replica symmetry breaking for spherical
  spin glasses.
\newblock {\em Proceedings of the American Mathematical Society},
  146(7):3127--3142, 2018.

\bibitem[KB05]{krzkakala2005disorder}
F~Krz{\k{a}}ka{\l}a and J-P Bouchaud.
\newblock Disorder chaos in spin glasses.
\newblock {\em Europhysics Letters}, 72(3):472, 2005.

\bibitem[K{\i}z23]{kizildag2023reply}
Eren~C K{\i}z{\i}lda{\u{g}}.
\newblock {Personal Communication}.
\newblock 2023.

\bibitem[KMRT{\etalchar{+}}07]{krzakala2007gibbs}
Florent Krzaka{\l}a, Andrea Montanari, Federico Ricci-Tersenghi, Guilhem
  Semerjian, and Lenka Zdeborov{\'a}.
\newblock Gibbs states and the set of solutions of random constraint
  satisfaction problems.
\newblock {\em Proceedings of the National Academy of Sciences},
  104(25):10318--10323, 2007.

\bibitem[KPV93]{kurchan1993barriers}
Jorge Kurchan, Giorgio Parisi, and Miguel~Angel Virasoro.
\newblock Barriers and metastable states as saddle points in the replica
  approach.
\newblock {\em Journal de Physique I}, 3(8):1819--1838, 1993.

\bibitem[MMZ05]{mezard2005clustering}
Marc M{\'e}zard, Thierry Mora, and Riccardo Zecchina.
\newblock Clustering of solutions in the random satisfiability problem.
\newblock {\em Physical Review Letters}, 94(19):197205, 2005.

\bibitem[Mon95]{monasson1995structural}
R{\'e}mi Monasson.
\newblock Structural glass transition and the entropy of the metastable states.
\newblock {\em Physical review letters}, 75(15):2847, 1995.

\bibitem[MPZ02]{mezard2002analytic}
Marc M{\'e}zard, Giorgio Parisi, and Riccardo Zecchina.
\newblock Analytic and algorithmic solution of random satisfiability problems.
\newblock {\em Science}, 297(5582):812--815, 2002.

\bibitem[MRZ16]{montanari2015limitation}
Andrea Montanari, Daniel Reichman, and Ofer Zeitouni.
\newblock On the limitation of spectral methods: From the gaussian hidden
  clique problem to rank one perturbations of gaussian tensors.
\newblock {\em IEEE Transactions on Information Theory}, 63(3):1572--1579,
  2016.

\bibitem[MS02]{milgrom2002envelope}
Paul Milgrom and Ilya Segal.
\newblock Envelope theorems for arbitrary choice sets.
\newblock {\em Econometrica}, 70(2):583--601, 2002.

\bibitem[MS06]{montanari2006rigorous}
Andrea Montanari and Guilhem Semerjian.
\newblock Rigorous inequalities between length and time scales in glassy
  systems.
\newblock {\em Journal of statistical physics}, 125:23--54, 2006.

\bibitem[Pan13]{panchenko2013sherrington}
Dmitry Panchenko.
\newblock {\em {The Sherrington-Kirkpatrick model}}.
\newblock Springer Science \& Business Media, 2013.

\bibitem[RM14]{richard2014statistical}
Emile Richard and Andrea Montanari.
\newblock {A statistical model for tensor PCA}.
\newblock {\em Advances in neural information processing systems}, 27, 2014.

\bibitem[Sel23]{sellke2023free}
Mark Sellke.
\newblock {Free energy subadditivity for symmetric random Hamiltonians}.
\newblock {\em Journal of Mathematical Physics}, 64(4), 2023.

\bibitem[Sel24]{sellke2023threshold}
Mark Sellke.
\newblock {The threshold energy of low temperature Langevin dynamics for pure
  spherical spin glasses}.
\newblock {\em Communications on Pure and Applied Mathematics}, 2024.

\bibitem[Sub17]{subag2017geometry}
Eliran Subag.
\newblock {The Geometry of the Gibbs Measure of Pure Spherical Spin Glasses}.
\newblock {\em Inventiones mathematicae}, 210(1):135--209, 2017.

\bibitem[Tal06a]{talagrand2006free}
Michel Talagrand.
\newblock Free energy of the spherical mean field model.
\newblock {\em Probability theory and related fields}, 134(3):339--382, 2006.

\bibitem[Tal06b]{talagrand2006parisimeasures}
Michel Talagrand.
\newblock Parisi measures.
\newblock {\em Journal of Functional Analysis}, 231(2):269--286, 2006.

\end{thebibliography}

\end{document}